\documentclass[10pt, article]{amsart}
\usepackage{geometry} 
\usepackage{ae} 
\usepackage[T1]{fontenc}
\usepackage[cp1250]{inputenc}
\usepackage{amsmath}
\usepackage{amssymb, amsfonts,amscd,verbatim}
\usepackage{mathtools}
\usepackage{MnSymbol}
\usepackage{tikz}
\usetikzlibrary{positioning}
\usepackage{graphicx}
\usepackage{tikz-cd}
\usepackage{comment}

\usepackage[normalem]{ulem}
\usepackage{indentfirst}
\usepackage{latexsym}
\input xy
\xyoption{all}

\usepackage[bookmarks=true,
              bookmarksnumbered=true, breaklinks=true,
              pdfstartview=FitH, hyperfigures=false,
              plainpages=false, naturalnames=true,
              colorlinks=true,pagebackref=true,
              pdfpagelabels]{hyperref}

\usepackage{amsmath}    
\newcommand{\ignore}[1]{ }
\newcommand{\G}[1]{G^{(#1)}}

\theoremstyle{plain}
\newtheorem{Pocz}{Poczatek}[section]
\newtheorem{Proposition}[Pocz]{Proposition}

\newtheorem{Theorem}[Pocz]{Theorem}

\newtheorem{Lemma}[Pocz]{Lemma}

\theoremstyle{definition}
\newtheorem{Observation}[Pocz]{Observation}
\newtheorem{Notation}[Pocz]{Notation}

\newtheorem{Example}[Pocz]{Example}
\newtheorem{Definition}[Pocz]{Definition}
\newtheorem{Remark}[Pocz]{Remark}

\newtheorem{theorem*}{Theorem}
\newtheorem{theorem}{Theorem}

\newtheorem{remark}{Remark}
\newtheorem{definition}{Definition}
\newtheorem{proposition}{Proposition}

\def\CC{{\mathbb C}}

\def\ZZ{{\mathbb Z}}
\def\NN{{\mathbb N}}
\def\TT{{\mathbb T}}

\def\interval{[0,1]}

\numberwithin{equation}{section}

\author{Kyle ~ Austin}
\thanks{The first author was funded by the Israel Science Foundation (grant No. 522/14).}
\address{Ben Gurion University of the Negev}
\email{ksaustin88@gmail.com}

\author{Atish ~ Mitra}
\address{Montana Tech of the University of Montana}
\email{atish.mitra@gmail.com}

\title[Groupoid Models of $C^*$-algebras and The  Gelfand Functor]%
  {Groupoid Models of $C^*$-algebras and The  Gelfand Functor}

\date{ \today
}


\includecomment{comment}
\begin{document}

\maketitle

\begin{abstract}
We construct a large class of morphisms, which we call partial morphisms, of groupoids that induce  $*$-morphisms of maximal and minimal groupoid $C^*$-algebras. We show that the assignment of a groupoid to its maximal (minimal) groupoid $C^*$-algebra and the assignment of a partial morphism to its induced morphism are functors (both of which extend the Gelfand functor). We show how to geometrically visualize lots of $*$-morphisms between groupoid $C^*$-algebras. As an application, we construct, without any use of the classification theory, groupoid models of the entire inductive systems used in the original constructions of the Jiang-Su algebra $\mathcal{Z}$ and the Razak-Jacelon algebra $\mathcal{W}$. Consequently, the inverse limit of the groupoid models for the aforementioned systems are models for $\mathcal{Z}$ and $\mathcal{W}$, respectively.
\end{abstract}

\section{Introduction}

Modeling $C^*$-algebras using groupoids and bundles over groupoids is a long established method by which one can investigate a $C^*$-algebra using ideas from geometric topology, representation theory, topological dynamics and so on. One of the most notable reasons why modeling $C^*$-algebras using groupoids is important is due to the monumental achievement of J. L. Tu in \cite{Tu} that the $C^*$-algebra of an amenable groupoid satisfies the UCT (see also \cite{BL}  by S. Barlak and X. Li for a generalization to twisted groupoid $C^*$-algebras). Another notable result is the remarkable characterization in \cite{Renault2} by J. Renault that all Cartan pairs are of the form $(C^*(G,\sigma), C_0(\G0))$  where $G$ is an \'etale groupoid and $\sigma:\G2\to \TT$ is a 2-cocycle (see also \cite{K}). In fact, while the current authors were working on this paper, X. Li in \cite{Li} showed that a unital, simple, separable, finite nuclear dimension $C^*$-algebra satisfies the UCT if and only if it has a Cartan subalgebra and hence every $C^*$-algebra in the Elliott classification program is a twisted groupoid $C^*$-algebra of an amenable groupoid. The aim of this paper is to provide a systematic approach to groupoid modeling of $C^*$-algebras by making it functorial (see Remark \ref{functoriality}) and geometric.

Let $\mathcal{G}$ denote the category of locally compact groupoids (equipped with a fixed Haar system of measures) with morphisms given by partial morphisms which will be explained in Section \ref{maintool}. $\mathcal{G}$ has a full subcategory $\mathcal{T}$ of locally compact Hausdorff spaces with partial proper continuous maps (see Section \ref{maintool} or Appendix A). Let $\mathcal{C}$ denote the category of $C^*$-algebras with $*$-morphisms and let $\mathcal{COMM}$ denote the full subcategory of $\mathcal{C}$ consisting of commutative $C^*$-algebras with $*$-morphisms. One of the main objectives of this paper is to prove the following. 

\begin{theorem}\label{majortheorem}
There are contravariant functors $\Gamma_{\text{max}}$ and $\Gamma_{\text{min}}$ from $\mathcal{G}$ to $\mathcal{C}$ such that
\begin{enumerate}
    \item $\Gamma_{\text{max}}$ and $\Gamma_{min}$ both respect finite sums and map inverse limits to direct limits.
    \item $\Gamma_{\text{max}}$ and $\Gamma_{\text{min}}$ both extend the Gelfand duality functor and, moreover, both are equivalences between $\mathcal{T}$ and $\mathcal{COMM}$ (see Remark \ref{GelfandRemark}).
\end{enumerate}
\end{theorem}

\begin{Remark}
We even prove a generalization of Theorem 1 to include cocycles, see Subsection \ref{twistedfunctor}.
\end{Remark}

Our motivation for proving Theorem \ref{majortheorem} was to prove the following Theorem ( proven in Section \ref{mainsection}).

\begin{Theorem}\label{JSandRJ}[Theorem \ref{JiangSu}, Theorem \ref{RJ}]
There exists inverse sequences of groupoids whose image under $\Gamma_{\text{min}}$ (or $\Gamma_{\text{max}}$) are exactly the inductive systems used in the construction of $\mathcal{Z}$ in \cite{JS} and $\mathcal{W}$ in \cite{J}. The inverse limits of these inverse sequences are groupoid models for $\mathcal{Z}$ and $\mathcal{W}$, respectively.
\end{Theorem}

\begin{Remark}\label{GelfandRemark}
Recall that the classical Gelfand duality theorem does \textbf{not} give an equivalence of the categories of locally compact spaces with proper continuous maps and the commutative $C^*$-algebras with $*$-morphisms. Only $*$-morphisms which map approximate units to approximate units can be modeled by pullbacks of proper continuous functions, see Appendix A for a full explanation with examples. Full equivalence using partially defined maps has been known for a long time, but we thought it prudent to point out that our functors $\Gamma_{min}$ and $\Gamma_{max}$ have the full equivalence built into them.
\end{Remark}

\begin{Remark}\label{functoriality}
Functoriality from groupoids to $C^*$-algebras has been noticed in the past by R. Holkar in \cite{H} (also see \cite{BS1} for a related concept), but the functoriality in this paper only seems to be related to these works in a parallel way. In general, morphisms of groupoids are not a special case of groupoid actions and actions of groupoids on other groupoids cannot be interpreted as functors between those groupoids. Therefore one cannot think of either concept as a subcase of the other. The concept of Haar system preserving morphism (see Definition \ref{haarpreserving}) does seem to have an uncanny resemblence to condtion (iv) of the definition of topological correspondences (Definition 2.1 in \cite{H}). 
\end{Remark}

\begin{Remark}\label{relatedthings}
The following is a list of consequences for $\Gamma_{\text{min}}$ and $\Gamma_{\text{max}}$.
\begin{enumerate}
    \item A. Buss and A. Sims in \cite{BS} show that $\Gamma_{\text{min}}$ and $\Gamma_{\text{max}}$ are not surjective  as every groupoid $C^*$-algebra is isomoprhic to its opposite algebra and there are $C^*$-algebras which are not isomorphic to their opposite algebras.
    \item By work \cite{EP} of  R. Exel and M. Pardo, the class of Katsura algebras (and hence also Kirchberg algebras that satisfy the UCT) are in the range of both $\Gamma_{\text{min}}$ and $\Gamma_{\text{max}}$ and hence the Kirchberg algebras are in the range of both $\Gamma_{\text{min}}$. $\Gamma_{\text{max}}$. 
    \item I. Putnam in \cite{P} shows that if $(G_0,G_1)$ are a pair of groups with $G_0$ a simple, acyclic dimension group and $G_1$ a countable, torsion-free, abelian group, then there exists  a minimal, amenable, \'etale equivalence relation R on a Cantor set whose associated $C^*$-algebra $\Gamma_{\text{min}}(R)$ is Elliott classifiable and, moreover has $K_1(\Gamma_{\text{min}}(R)) = G_1$ and $K_0(\Gamma_{\text{min}}(R))= G_0$.
    \item More generally, X. Li in \cite{Li} shows that all Elliott classifiable $C^*$-algebras are in the range of $\Gamma^{\sigma}_{\text{min}}$ (and also for $\Gamma_{\text{max}}^\sigma$).
    \item In \cite{DPS}, R. Deeley, K. Strung, and I. Putnam also construct a groupoid model for $\mathcal{Z}$ using topological dynamics; i.e. they show that $\mathcal{Z}$ is in the image of both $\Gamma_{\text{min}}$ and $\Gamma_{\text{max}}$.
\end{enumerate}
\end{Remark}

In the process of completing this paper, X. Li in \cite{Li} posted a paper that shows that $\mathcal{Z}$ and $\mathcal{W}$ are twisted \'etale groupoid $C^*$-algebras and some of his techniques are quite similar to ours. His method of modeling comes from the works \cite{BL} and \cite{BL2}. Some notable differences with the modeling done in those works with the modeling done in the current work are
\begin{enumerate}
    \item We work with both minimal and maximal completions of groupoid convolution algebras whereas they work only with reduced completions.
    \item Our modeling works for groupoids with general Haar systems of measures and not just for \'etale groupoids.
    \item Our concept of partial morphisms allows us to model exactly lots of morphisms of $C^*$-algebras whereas they have to break $*$-morphisms into two pieces  and model each piece separately.
\end{enumerate}

Our groupoid models for $\mathcal{Z}$ and $\mathcal{W}$ are isomorphic to the models constructed in \cite{Li}, but it is worth noting that both our models differ considerably to the models of R. Deeley, I. Putnam, and K. Strung in \cite{DPS}. One major difference is that we make no use of the classification theory whereas the authors in \cite{DPS} make critical use of it (see the proof of Proposition 2.8 in \cite{DPS} for example).

\section{Preliminaries on Groupoid \texorpdfstring{$C^*$}{C*}-Algebras}

Recall that a \emph{groupoid} is a small category in which  every morphism is invertible. As is standard througout the literature, we adopt the ``arrows only'' view of category theory so we view a groupoid $G$ as just a collection of arrows. We denote the object space of $G$ by $ \G0$ and the arrow space of $G$ by $\G1$. 

\begin{Notation}
We will denote arrows of groupoids using the letters $x,y$ and $z$ and we will use $u, v$ and $w$ to denote objects. 
\end{Notation}

Recall that there are maps $s$ and $r$ mapping $G$ to $\G0$ by taking each arrow to its source or range, respectively. For each $u\in \G0$, we let $G^u$ denote the fiber of $u$ under the range map and we let $G_u$ denote the fiber of $u$ under the source map; i.e. $ G^u= \{x\in G:r(x)=u\}$ and $G_u= \{x\in G:s(x)= u\}$.  A \emph{topological groupoid} is a groupoid $G$ equipped with locally compact and Hausdorff topology that makes the composition of arrows and the inversion of arrows continuous and, furthermore, carries a fixed Haar system of measures. Recall that a \emph{Haar system of measures} for $G$ is a collection $\{\mu^u:u\in \G0\}$ of positive Radon (see Remark \ref{Radon}) measures on $G$ such that 
\begin{enumerate}
\item $\mu^u$ is supported on $G^u$.
\item  For each $f\in C_c(G)$, the function $u\to \int_{G}f(x)d\mu^u(x)$ is continuous on $\G0$.
\item For all $x\in \G1$ and $f\in C_c(G)$, the following equality holds: 
$$\int_{G}f(y)d\mu^{t(x)}(y) = \int_{G}f(xy)d\mu^{s(x)}(y).$$
\end{enumerate}

\begin{Remark}\label{Radon}
Because the concept of Radon measures does not seem to be totally agreed upon throughout the literature, we define a \textbf{Radon measure} $\mu$ on a locally compact Hausdorff space $X$ to be locally finite Borel measure such that, for every measureable set $A\subset X$ with $\mu(A)<\infty$ and for every $\epsilon>0$, there exists a compact subset $K\subset A \subset X$ and an open subset $A\subset U\subset X$ such that $\mu(A\setminus K)<\epsilon$ and $\mu(U\setminus A)<\epsilon$. Recall that if $X$ is $\sigma$-compact then the requirement $\mu(A)<\infty$ can be removed from the above condition. 

By the Riesz-Markov-Kakutani theorem, the collection of positive Radon measures on a locally compact space $X$ are in one-to-one correspondence with positive functionals $C_c(X)\to \CC$.
\end{Remark}

A groupoid $G$ is \textbf{principal} if the collection $G(u) = \{x\in G: s(x) = t(x) = u\}$  contains only one element, namely $u$, for every $u\in \G0$. Recall from \cite{Buneci} that a locally compact and principal groupoid $G$ is isomorphic to an equivalence relation on its object space $\G0$, with the minor modification that the topology one has to take on $\G0\times \G0$ may be finer than the product topology.
A groupoid $G$ is said to be \textbf{r-discrete} if the object space $\G0$ is open in $G$. Recall that a groupoid $G$ is \textbf{\'etale} if one of the following equivalent conditions holds, see Proposition 2.8 in Chapter 1 of (\cite{Renault})
\begin{enumerate}
    \item The range map is a local homeomorphism
    \item $G$ is $r$-discrete and has an open range map.
    \item $G$ is $r$-discrete and has any Haar system of measures.
    \item $G$ is $r$-discrete and counting measures form a Haar system for $G$.
\end{enumerate}

Let $G$ be a  topological groupoid with Haar system of measures $\{\mu^u:u\in \G0\}$, a function $f\in C_c(G)$ defines a continuous field of kernels $K_u(x,y):G^u\times G^u\to \CC$ where $k_u(x,y) = f(xy^{-1})$. The following formula shows how to interpret the convolution product $f*g$ by integrating $g\in C_c(G)$ against the continuous field of kernels associated to $f$.

$$ f*g(x) = \int_{G} f(xy^{-1})g(y)d\mu^{s(x)}(y) $$ 

 With $*$ as multiplication and the following  adjoint operation and norm $\|\cdot\|_I$, $C_c(G)$ becomes a topological *-algebra (with inductive limit topology).

$$ f^*(x) = \overline{f(x^{-1})}$$

$$\|f\|_I = \text{max} \left\{\sup_{u \in \G0} \int_{G} |f(g)|\,d\mu^u(g), \sup_{u \in \G0} \int_{G} |f(g^{-1})|\,d\mu^u(g)\right\}.$$

The \textbf{maximal (or full) groupoid $C^*$-algebra of $G$}, denoted by $C^*(G)$, is defined to be the completion of $C_c(G)$ with the following norm
$$\|f\|_{\text{max}} = \sup_{\pi} \|\pi(f)\|,$$
where $\pi$ runs over all continuous (with respect to the norm $\|\cdot\|_I$) *-representations of $C_c(G)$.  The \textbf{minimal (or reduced groupoid $C^*$-algebra} $C^*_r(G)$ is defined to be the smallest completion of $C_c(G)$ such that all the representations $\{\text{Ind}(\delta_u):u\in \G0\}$ are continuous where $\text{Ind}(\delta_u)$ is the representation defined as follows: Firstly, we use the standard method of turning the Dirac measures $\delta_u$ for $u\in \G0$ into representations by first defining the new measure $\theta_u$ to be the unique positive Radon measure on $G$ defined by 

$$\int_G f d\theta_z = \int_{\G0} \int_G f d\mu^wd\delta_u(w).$$

Now, we define a representation $\text{Ind}(\delta_z):C_c(G) \to B(L^2(G,\theta_z))$ by 

$$\text{Ind}(\delta_u)(f)(\xi)(x) = \int_{G} f(xy^{-1})\xi(y)d\theta_u(y)$$

The reduced norm $||\cdot||_r$ on $C_c(G)$ is defined by $||f||_r = \sup_{u\in \G0} ||\text{Ind}(\delta_u)(f)||$. Note that $||\cdot ||_r \leq ||\cdot||_{\text{max}}$. 

\begin{Definition}\label{haarpreserving}
Let $G$ and $H$ be locally compact groupoids with Haar systems of measures $\{\mu^u:u\in \G0\}$ and $\{\nu^v:v\in H^{(0)}\}$.  A \textbf{groupoid morphism} $\phi:G\to H$ is a proper, continuous, and covariant (see Remark \ref{contravariant}) functor. $\phi$ is said to be \textbf{Haar system preserving} if, for all $v\in H^{(0)}$ and for all $u \in \phi^{-1}(v)$, we have that  $\phi_*\mu^u = \nu^v$.
\end{Definition}

\begin{Remark}\label{contravariant}
One could also work with with contravariant functors, however one would just need to incorperate the modular function into the induced morphisms (as a contravariant functor $G\to H$ is equivalent to a covariant functor $G^{op}\to H$). For sake of keeping this paper easier to read, we decided to leave this case out.
\end{Remark}

\begin{Remark}\label{specialcase}
We make the following observations about the geometric nature of this seemingly analytic nature of Haar system preserving morphisms. 
\begin{enumerate}
    \item The Haar system preserving condition looks purely analytical, but the necessity of the condition is really just a consequence of the fact that groupoids inherit a lot of structure from the range (or source) fibration. For example, the convolution product relies entirely on the structure of the natural action of the groupoid on itself as a bundle of measure spaces over the range map. The Haar system preserving condition really just says that the morphism must preserve the continuous field of measure space nature of groupoids with Haar systems; i.e.  they are maps of bundles that must send fibers to whole fibers. In Section \ref{maintool}, we will show how to make this concept more flexible by considering only partially defined functions.
    \item The first author and M. Georgescu show in \cite{AG} that any $\sigma$-compact groupoid admits enough Haar system preserving morphisms to second countable groupoids to completely determine its original topological groupoid structure. In particular, they show that there are lots of nontrivial examples of Haar system preserving morphisms.
    \item If $G$ and $H$ are \'etale groupoids then $\phi$ is Haar system preserving if and only if for each $u\in \G0$, $\phi$ bijectively maps the set $G^u$ to $H^{\phi(u)}$. 
\end{enumerate}
\end{Remark}

The following comes from Proposition 3.2 in \cite{AG}, but we decided to add the proof for completeness.

\begin{Proposition}\label{inducedmorphism} 
Let $q:G\to H$ be a Haar system preserving morphism of locally compact Hausdorff groupoids with Haar systems $\{\mu^u:u\in \G0\}$ and $\{\nu^v:v\in H^{(0)}\}$, respectively. The pullback map $q^*:C_c(H)\to C_c(G)$ is a *-morphism of topological *-algebras and is $I$-norm decreasing.  If additionally $q$ is surjective, then $q^*$ is I-norm preserving and hence extends to an isometric $*$-embedding with respect to maximal completions. 
\end{Proposition}
\begin{proof}

The usual pullback of $\phi$ induces a $\CC$-module  homomorphism $\phi^*: C_c(H)\to C_c(G)$. We claim that it is a $*$-algebra morphism.  Let $f_1, f_2 \in C_c(H)$ and $y \in G$. We define $F_1, F_2 \in C_c(H)$ by $F_1(h) = f_1(\phi(y)h)$ for $h\in \phi(s(y))$  (and 0 elsewhere) and $F_2(h) = f_2(h^{-1})$ for $h \in H$.   We have

\begin{align*}
(\phi^*(f_1) \ast \phi^*(f_2))(y) &= \int_G f_1(\phi(yx))f_2(\phi(x^{-1}))\,d\mu^{s(y)}(x) \\
&= \int_G (F_1 \cdot F_2)(\phi(x))\,d\mu^{s(y)}(x) \\
&= \int_H (F_1 \cdot F_2)(z)\,d\nu^{\phi(s(y))}(z) \\
&= \int_H f_1(\phi(y)z)f_2(z^{-1})\,d\nu^{\phi(s(y))}(z) \\
&= (f_1 * f_2)(\phi(y)) = (\phi^*(f_1 \ast f_2))(y).
\end{align*}

An even easier computation shows that $\phi^*$ preserves the adjoint and hence is a *-morphism. We will leave the simple proof that $\phi$ is $I$-norm decreasing to the reader. 
\end{proof}

\begin{Proposition}\label{reduced}
Let $\phi:G\to H$ be a Haar system preserving morphism of locally compact Hausdorff groupoids with Haar systems $\{\mu^u:u\in \G0\}$ and $\{\nu^v:v\in H^{(0)}\}$, respectively. The pullback map $q^*:C_c(H)\to C_c(G)$ is a *-morphism of topological *-algebras that extends to reduced completions. If $\phi$ is surjective, then $\phi^*$ is isometric with respect to the reduced norms $||\cdot||_r$
\end{Proposition}
\begin{proof}
Let $u\in \G0$ and notice that for all $g\in C_c(H)$, we have 
$$\int_{\G0}\int_{G} \phi^*(g)(z) d\mu^u(z) d\delta_x(u) = \int_{H^{(0)}}\int_H gd\nu^{v}(y)d\delta_{\phi(x)}(v)$$
because $\nu^{\phi(w)}$ is the pushforward measure of $\mu^w$ for every $w\in \G0$ and because $\delta_{\phi(x)}$ is the pushforward measure of $\delta_x$. It therefore follows that the representation $\text{Ind}(\delta_{\phi(u)})$ is the composition $\text{Ind}(\delta_u)\circ \phi^*$. It follows that $\phi^*$ extends to reduced completions. 

The fact that $\phi^*$ is isometric if $\phi$ is surjective follows from the fact that for every $u\in H^{(0)}$  the measure $\theta_u$ is a pushforward of $\theta_v$ for some $v\in \G0$.  
\end{proof}

The following proposition, proven in \cite{AG}, shows that there are indeed lots of morphisms of groupoids that preserve Haar systems.

\begin{Proposition}\label{Exampleshaarmeasurepreserving}
Let $G$ be a topological groupoid with Haar system and let $f:X\to Y$ be a proper continuous function of locally compact Hausdorff spaces. Then the morphism $\text{id}_G\times f: G\times X\to G\times Y$ is Haar system preserving (here, we take the Haar system to be $\delta_x\times \mu^y$ for $(x,y)\in \G0\times X$). 
\end{Proposition}

The following are some examples of Proposition \ref{Exampleshaarmeasurepreserving}

\begin{itemize}
    \item Suppose that $G$ is a groupoid with Haar system and $X$ is a compact Hausdorff space. Then the coordinate projection $G\times X \to G$ is Haar system preserving and, furthermore, has pullback map given by $C^*(G) \to C^*(G)\otimes C(X)$ given by $a\to a\otimes 1$.
    \item Suppose that $f:X\to Y$ is a proper continuous map of locally compact Hausdorff spaces and let $G_n$ be as in Example \ref{matrixgroupoid} then the map $f\times \text{id}_{G_n}:X\times G_n\to Y\times G_n$ is Haar system preserving and, furthermore, the pullback map induces the morphism $M_n(C_0(Y)) \to M_n(C_0(X))$ where $(g_{i,j}) \to (f^*(g_{i,j}))$.
\end{itemize}

In most cases in this paper, it is straightforward to see that a particular map is Haar system preserving by simply checking the definition.  Even though the definition seems quite rigid, all of the morphisms we use in this paper can be easily checked to be Haar system preserving.

\begin{Definition}
Let $G$ and $H$ be locally compact Hausdorff groupoids with Haar systems of measures $\{\mu^u:u\in G^0\}$ and $\{\lambda^v:v\in H^0\}$. We define the \textbf{product of $G$ and $H$}, denoted by $G\times H$, to be the groupoid with object space $G^0\times H^0$ and with arrows $G^1 \times H^1$. We give $G\times H$ the Haar system  $\{\mu^u \times \lambda^v:(u,v)\in G^0\times H^0\}$. 
\end{Definition}

\begin{Lemma}\label{producttotensor}
Let $G$ and $H$ be locally compact Hausdorff  groupoids with Haar systems of measures $\{\mu^u,u\in \G0\}$ and $\{\lambda^v:v\in H^{
(0)}\}$, respectively. Then $C^*(G\times H)\cong C^*(G) \otimes_{\text{max}} C^*(H)$ and $C_r^*(G\times H)\cong C_r^*(G) \otimes_{\text{min}} C_r^*(H)$. 
\end{Lemma}
\begin{proof}
Consider the usual topological vector space isomorphism $\phi:C_c(G)\otimes C_c(H) \to C_c(G\times H)$ (with inductive limit topologies) given by $\phi(f\otimes g)((x,y)) = f(x)g(y)$ for simple tensors and extended linearly over sums of simple tensors. 
To see that $\phi$ is convolution product preserving, let $f_i\in C_c(G)$ and $g_i\in C_c(H)$ for $i=1,2$ and consider 

\begin{align*}
\phi(f_1\otimes g_1 * f_2\otimes g_2)((x,y)) & =  \phi(f_1*f_2 \otimes g_1 * g_2)((x,y)) \\
& = \left(\int_G f_1(z^{-1})f_2(xz)d\mu^{s(x)}(z)\right)\left(\int_H g_1(w^{-1})g_2(yw)d\lambda^{s(y)}(w)\right)\\
& = \int_G \int_H f_1(z^{-1})f_2(xz)g_1(w^{-1})g_2(yw)d\mu^{s(y)}(w) d\lambda^{s(x)}(z) \\
& = \int_{G\times H} f_1(z^{-1})f_2(xz)g_1(w^{-1})g_2(yw)d\mu^{s(y)}(w) d\left(\mu^{s(x)}\times \lambda^{s(y)}\right)(w,z) \\
& = \phi(f_1\otimes g_1) * \phi(f_2\otimes g_2) (x,y)
\end{align*}

Notice that the above calculation extends to sums of simple tensors. A straightforward calculation shows that $\phi$  preserves adjoints. 

To see that the maximal completion of $C_c(G\times H)$ is isomorphic to $C^*(G)\otimes_{\text{max}}C^*(H)$, one just needs to notice that $I$-norm continuous morphisms of $C_c(G\times H)$ are in one-to-one correspondence with commuting  $I$-norm continuous representations of $C_c(G)$ and $C_c(H)$.

As for the reduced case, notice that the representations $\text{Ind}(\delta_{(x,y)}) = \text{Ind}(\delta_x\times\delta_y)$ correspond exactly to the representation $\text{Ind}(\delta_x)\otimes \text{Ind}(\delta_y)$ on the Hilbert space $L^2(G,\theta_x)\otimes L^2(H,\theta_y)$. This corresponds exactly to the completion $C_r^*(G) \otimes_{\text{min}} C_r^*(H)$.
\end{proof}

\section{Partial Morphisms, Induced Maps, and Inverse Systems}\label{maintool}

\begin{Definition}\label{opensubgroupoid}
Let $G$ be a topological groupoid with Haar system of measures and suppose that $K\subset G$ is a open subgroupoid. Then $K$ carries a Haar system of measures so that the extension of $f\in C_c(K)$ to $G$ by setting $f(g) = 0$ for all $g\in G\setminus K$ is a $*$-morphism of convolution algebras.
An open subgroupoid $K\subset G$ with this Haar system of measures will be referred to as a \textbf{Haar subgroupoid}.
\end{Definition}

When working with inductive systems of $C^*$-algebras, one often requires that the bonding maps be unital or that a particular bonding  map lands in a particular subalgebra. As mentioned in the introduction, one cannot model the morphism $M_k\to M_k\otimes M_n$ given by $T\to T\otimes \text{id}_n$ with pullback maps (in fact, the reader can easily verify this fact on their own). The following concept is designed precisely to make up for the failings of pullback maps, by allowing more flexibility with their domains.

\begin{Definition}\label{partialmorphisms}
Let $G$ and $H$ be topological groupoid with Haar systems of measures. A 
\textbf{partial morphism} from $G$ to $H$ is a pair $(f,K)$ where $K$ is a Haar subgroupoid of $G$ ($K$ may be all of $G$) and $f:K\to H$ is a Haar system preserving morphism of groupoids. 
\end{Definition}

\begin{Notation}
Sometimes we will write $f:G\to H$ as a partial morphism and suppress the $K$, especially when $K$ is not relevant to the point we are making. 
\end{Notation}

\begin{Remark}
A Haar system preserving morphism of groupoids is a partial morphism and corresponds to the case where $G=K$ in the above definition. Also note that $f=\text{id}_K$ corresponds to the inclusion of $C_c(K)\hookrightarrow C_c(G)$, but we have artificially made this inclusion contravariant.
\end{Remark}

\begin{Remark}
Partial morphisms are motivated by the concepts of subhomomorphisms or local homomorphisms of skew fields. In fact, our idea of taking inverse limits with partial morphism bonding maps is closely related to the work of A. I. Lichtman in \cite{L}. 
\end{Remark}

\begin{Definition}\label{generalizedpullback}
Let $G$ and $H$ be groupoids with Haar systems of measures and let $(f,K):G\to H$ be a partial morphism. We define the \textbf{induced map} to be the composition 
$$ C_c(H) \overset{f^*}\to  C_c(K) \xrightarrow{\text{canonical}} C_c(G)$$
where the canonical inclusion $C_c(K) \to C_c(G)$ was outlined in Definition \ref{opensubgroupoid}.
\end{Definition}

\begin{Remark}
Note that if $\phi:G\to H$ is a  partial morphism then by Propositions \ref{inducedmorphism} and \ref{reduced}, the induced map $\phi^*$ is a $*$-morphism of minimal (maximal) groupoid $C^*$-algebras. If, additionally, the partial morphism is surjective, then the induced $*$-morphism is isometric (for both the minimal and maximal completions). 
\end{Remark}

The following proposition is one of the key reasons why we are looking at partial morphisms and their induced maps in this paper. We omit the straightforward proof.

\begin{Proposition}\label{reasonforgenpullbacks}
Let $G$ and $H$ be \'etale groupoids with $\G0$ compact and let $\phi:G\times H \to H$ be the partial map defined by projection of the open subgroupoid $\G0\times H \to H$. The induced morphism, composed with the isomorphism $C^*(G\times H) \to C^*(G)\otimes C^*(H)$ (which is the inverse of the isomorphism defined in the proof of Proposition $\ref{producttotensor}$ ) is the morphism $C^*(H) \hookrightarrow C^*(G)\otimes C^*(H)$ given by $a \to 1\otimes a$. The result is also true if one takes minimal completions.
\end{Proposition}

Let $\mathcal{G}$ denote the category whose objects are locally compact Hausdorff groupoids with Haar systems of measures and with partial morphisms. Notice that if $\phi:G\to H$ and $\psi:H\to K$ are partial morphisms then the domain of $\psi$ is an open  subset of $H$ and hence its preimage is open in $G$. Moreover, because $\phi$ is Haar system preserving, it follows that the pre-image of the domain of $\psi$ via $\phi$, which we will denote by $A$, is a Haar subgroupoid of the domain of $\phi$ and hence is a Haar subgroupoid of $G$. We define the composition of $\phi$ and $\psi$ to be $\psi \circ \phi|_{A}$. $\mathcal{G}$ has a full subcategory $\mathcal{T}$ consisting of locally compact and Hausdorff spaces and with partial morphisms between spaces. 

It is straightforward to check that the assignment of a groupoid $G$ to its convolution algebra $C_c(G)$ and to a partial morphism $\phi:G\to H$ the induced map $\phi^*:C_c(H)\to C_c(G)$ is a functor from $\mathcal{G}$ to the category of topological $*$-algebras with continuous and $*$-preserving algebra morphisms. Extending to the maximal or minimal completions, we get the following key result.

\begin{Theorem}\label{functor}
Let $\mathcal{C}$ denote the category of $C^*$-algebras with $*$-morphisms. The association $\Gamma_{\text{max}}$ ($\Gamma_{\text{min}}$) which takes each object of $\mathcal{G}$ and assigns its maximal (minimal) groupoid $C^*$-algebra and assigns partial morphisms to their induced map is a contravariant functor. Moreover, both $\Gamma_{\text{max}}$  and $\Gamma_{\text{min}}$  are extensions of the Gelfand duality functor $\Gamma_c$.
\end{Theorem}

\begin{Remark}\label{dualityorno}
The reader may be thinking whether $\Gamma_{\text{min}}$ is a categorical equivalence like it is for commutative $C^*$-algebras ( see Subsection \ref{gelfand}). Unfortunately, we have $\Gamma_{\text{min}}(\TT)$ and $\Gamma_{\text{min}}(\ZZ)$ are both $C(\TT)$ where we view $\TT$ as a compact Hausdorff space and $\ZZ$ as a discrete group. The same example works for $\Gamma_{\text{max}}$. 
\end{Remark}

\begin{Definition}\label{geninvsystem}
We call a  diagram $\{G_\alpha, p^\alpha_\beta,\Delta\}$ in $\mathcal{G}$ an \textbf{inverse system of groupoids} if
\begin{enumerate}
    \item $\Delta$ is a directed set and for all $\alpha,\beta\in \Delta$ there exists $\gamma\in \Delta$ with $\gamma\ge \beta,\alpha$ and $p^\gamma_\alpha:G_\gamma\to G_\alpha$ and $p^\gamma_\beta:G_\gamma\to G_\beta$ are both in $\{G_\alpha, p^\alpha_\beta,\Delta\}$.
    \item $p^\alpha_\beta:G_\alpha \to G_\beta$ is a surjective partial morphism in $\mathcal{C}$.
    \item $p_\alpha^\alpha = \text{id}_{G_\alpha}$ for all $\alpha.$
    \item $p^\beta_\gamma \circ p^\alpha_\beta  = p^\alpha_\gamma$ for all $\alpha \ge \beta\ge \gamma.$
\end{enumerate}
\end{Definition}

We will refer to the maps $p^\alpha_\beta$ in Definition \ref{geninvsystem} as \textbf{bonding maps}. If $G$ is the inverse limit of an inverse system $\{G_\alpha, p^\alpha_\beta,\Delta\}$ then we refer to the partial maps $q_\alpha:G\to G_\alpha$ given by the universal property as \textbf{projections}. 

\begin{Example}
Let $\{X_\alpha,p^\alpha_\beta, \Delta\}$ be an inverse system of locally compact and Hausdorff spaces in which the partial maps $p^\alpha_\beta$ have domain $X_\alpha$ for each $\alpha\ge \beta$. It follows from the classical results that the inverse limit exists and is exactly the subspace of the product space $\Pi_{\alpha}X_\alpha$ consisting of those tuples $(x_\alpha)_{\alpha\in \Delta}$ for which $p^\alpha_\beta(x_\alpha) = x_\beta$. Such tuples are referred to as \textbf{threads}.
\end{Example}

The following is part of Theorem A in \cite{AG} and is a generalization of the previous example.

\begin{Example}\label{theorema}
Let $\{G_\alpha, \{\mu_\alpha^u:u\in G^0_\alpha\}),q^\alpha,\beta,A\}$ be an inverse system of groupoids with Haar systems  and with proper, continuous, surjective, Haar system preserving  bonding maps. Theorem A of \cite{AG} shows that if the domain of $p_\beta^\alpha$ is $G_\alpha$ for all $\alpha\ge \beta$ then the inverse limit groupoid $G= \varprojlim_{\alpha}G_\alpha$ exists and has a Haar system of measures $\{\mu^u:u\in G^0\}$ such that $(p_\alpha)_*(\mu^u) = \mu_\alpha^{p_\alpha(u)}$.
\end{Example}

\begin{Example}
Let $\{X_\alpha,i^\alpha_\beta,\Delta\}$ be a collection of locally compact Hausdorff spaces, $\Delta$ a directed set, and $i_\beta^\alpha:X_\beta \hookrightarrow X_\alpha$ an inclusion of an open subset. We form an inverse system $\{X_\alpha,p^\alpha_\beta,\Delta\}$ with partial morphisms embeddings $p_\alpha^\beta:X_\alpha\to X_\beta$ by letting the domain of $p_\beta^\alpha$ be $i_\beta^\alpha(X_\beta)$ and letting $p_\beta^\alpha: i_\beta^\alpha(X_\beta) \to X_\beta$ being the identity map. We leave to the reader to verify that the union $\bigcup_\alpha X_\alpha$ is the inverse limit of $\{X_\alpha,p^\alpha_\beta,\Delta\}$ where the partial projection maps are given by $p^\alpha_\beta$.
\end{Example}

The following example shows how inverse limits in $\mathcal{G}$ are actually unions of inverse limits.

\begin{Example}\label{Example3.14}
Let $X_0 = \{0,1\}$, $X_1 = X_0\times X_0 \times X_0$, and $p^1_0:X_1\to X_0$ be the partial morphism given by projection onto the first factor with domain $\{(x,y,0):x,y\in X_0\}$. In general, let $X_n = \Pi_{i=1}^{2n+1} X_0$ and let $p_n^{n+1}:X_{n+1}\to X_n$ have domain $U_n^{n+1}$ consisting of all points whose $2n+1$-coordinate is 0 and let $p_n^{n+1}$ be the standard projection to $X_n$. 

What does the inverse limit of such a sequence look like?  Fix $k\ge 0$ and define $V^n_k = (p^{n}_k)^{-1}(X_k) \subset U_k^n$. Notice that $\{V_k^{n}, p^{n+1}_{n}|_{V_k^n},n\ge k\}$ is an inverse sequence of compact metric spaces with surjective bonding maps, hence it has an inverse limit $Z_k$. Notice that $Z_k$ is a Cantor set for every $k\ge 0$ and, furthermore, for every $l\ge k$, $Z_k$ embeds into $Z_l$ as an open subset. We can explicitly describe this embedding actually. Recall that the Cantor set $C$ is $X_0^{\NN} \cong X_0^\NN\times X_0^\kappa$ for every countable cardinal $\kappa$. Notice that $Z_1$ can be viewed as $Z_0 \times \{0,1\}$ and that $Z_0$ embeds into $Z_1$ as those points whose second coordinate is 0. In general, $Z_k$ can be viewed as $Z_0 \times X_0^k$ and embeds into $Z_l = Z_0\times X_0^k \times X_0^{l-k}$ as those tuples whose $X_0^{l-k}$-coordinate is $(0,0,\hdots 0)$. 

We claim that $Z= \bigcup_{k}Z_k$ is the inverse limit of the sequence $\{X_n,p^{n+1}_n\}$. Observe that $Z_k$ is open in $Z$ for every $k\ge 0$ and that we have a canonical projection mapping $p_k$ from $Z_k$ to $X_k$ that commutes with the inverse system for every $k\ge 0$, namely, the inverse limit projection coming from the universal property for the inverse system $\{V_k^{n}: p^{n+1}_{n}\}$. Suppose that $\{Y,q_k\}$ was another space with partial morphisms $q_k$ mapping an open subset $U_k\subset Y$ to $X_k$ that commuted with with the partial mappings of the inverse system. Using the universal property of inverse limits for $\{V_k^n,p^{n+1}_n|_{V_k^n}\}$, it is easy to see that there are mappings $q'_k:U_k\to Z_k$ such that $p_k\circ q'_k = q_k$. Evidently, $Z$ is the inverse limit of $\{X_n,p^{n+1}_n\}$ with projection mappings $p_n$. 
\end{Example}

The following theorem subsumes all the previous examples and shows that inverse limits always exist.

\begin{Theorem}\label{mainthm.limitexists}
Every inverse system $\{G_\alpha,(p^\alpha_\beta,U^\alpha_\beta),
\Delta\}$ in $\mathcal{G}$ has an inverse limit $G$  in  $\mathcal{G}$. 
\end{Theorem}


\begin{proof}
We will first show that there is a locally compact space $G$ that is the inverse limit of the underlying spaces of the inverse system. Once this is done, it is fairly straightforward to put a groupoid structure  and a Haar system of measures on $G$.

Following our notation for Example \ref{Example3.14}, we let $V^\alpha_\beta = (p^{\alpha}_\beta)^{-1}(G_\beta)$ and notice that it is open in $U_\beta^\alpha$. It is straightforward to see that $\{V_\beta^{\alpha}, p^{\alpha}_{\beta}|_{V_\beta^\alpha}, \alpha \ge \beta\}$ is an inverse system of locally compact Hausdorff spaces with proper and surjective bonding maps, hence it has an inverse limit $Z_\beta$ which is itself a locally compact Hausdorff space (this uses the proper-ness of the bonding maps). Visually speaking, $Z_\alpha$ can be viewed as the space of threads which start at $G_\alpha$. Let $q^\beta_\alpha:Z_\beta\to G_\alpha$ denote the projections coming from the universal property of inverse limits for $\beta \leq \alpha$.

Notice that for $\beta\leq \alpha$ we have $Z_\beta$ embeds into $Z_\alpha$ as an open subset; indeed, $Z_\beta$ is exactly the inverse image of $U_\beta^\alpha$ under the inverse limit projection $q^\alpha_\beta:Z_\alpha\to G_\alpha$. Let $i^\alpha_\beta:Z_\beta \hookrightarrow Z_\alpha$ denote the open inclusion. It follows that $\{Z_\alpha,i_\beta^\alpha,\Delta\}$ is an inductive system of open inclusions.   Letting $G = \bigcup_{\alpha}Z_\alpha$ be the direct limit of this inductive system, we observe that  $G$ is locally compact and Hausdorff as it is an increasing union of open locally compact and Hausdorff spaces.

Observe that the projections $q_\beta^\alpha:Z_\alpha\to G_\beta$ are partial morphisms from $G$ to the system $\{G_\alpha,p^\alpha_\beta,\Delta\}$ that commute with the inverse system. Now, suppose  $Z$ is a locally compact Hausdorff space with partial maps $(q_\alpha,U_\alpha): Z \to G_\alpha$ such that $q^\alpha_\beta \circ q_\alpha = q_\beta$ for all $\alpha\ge \beta$. The reader should first notice that for each $\alpha \leq \beta$ we must have $U_\alpha\subset U_\beta$; indeed, we have $p^\alpha_\beta \circ q_\alpha = q_\beta$ and so $q_\beta^{-1}(G_\beta)= U_\alpha \subset q_\alpha^{-1}(G_\beta) = U_\beta$. It follows that, for all $\alpha\ge\beta$ there exists, via restriction of $q_\alpha$ to $U_\beta$,  maps $q_{\beta,\alpha}:U_\beta\to G_\alpha$ and, moreover, the image of $U_\beta$ under $q_{\beta,\alpha}$ lies in $V_\beta^\alpha$. It is easy to see that the maps $q_{\beta,\alpha}$ commute with the inverse system $\{V_\beta^{\alpha}, p^{\alpha}_{\beta}|_{V_\beta^\alpha}, \alpha \ge \beta\}$. Hence, by the universal property of the inverse limit $Z_\beta$, there exists a unique map $r_\beta:U_\beta \to Z_\beta$ such that $p_\beta \circ r_\beta =q_\beta$. We therefore have a partial map $(\bigcup_\alpha r_\alpha, \bigcup_\alpha U_\alpha):Z\to G$ that commutes with the projections $q_\alpha$ and $p_\alpha$ for all $\alpha$. It is unique as it is unique when restricted to $U_\alpha$ for each $\alpha$.

As $G$ is an increasing union of bona fide inverse limits of inverse systems of groupoids, each satisfying the hypothesis of Example \ref{theorema}, we know that $G$ is a union of open locally compact Hausdorff subgroupoids, each with a Haar system. By Proposition \ref{getthemmeasurestoworkout}  in Appendix B, there exists a collection of measures $\{\mu^x:x\in \G0\}$ of Radon measures with $\mu^x$ supported only on $G^x$ for every $x\in \G0$ and whose continuity and left invariance properties follow directly from the corresponding properties of the Haar systems of each of the open subgroupoids. More specifically, if $f \in C_c(G)$, then as $G$ is the increasing union of open subgroupoids, the support of $f$ lies in one of those open subgroupoids, where the continuity and left invariance is already known.
\end{proof}

\begin{Proposition}\label{equivrelation}
Let $(G_\alpha,p^\alpha_\beta,A)$ be a generalized inverse system of locally compact Hausdorff groupoids with fixed Haar systems $\{\mu_\alpha^u:u\in\G0_\alpha\}$. If $G_\alpha$ is an equivalence relation for each $\alpha$, then $G:=\varprojlim G_\alpha$ is an equivalence relation. Moreover, if $G_\alpha$ is \'etale for each $\alpha$ then $G$ is \'etale.
\end{Proposition}

\begin{proof}

As we have noted in the proof of Theorem \ref{mainthm.limitexists}, inverse limits of generalized inverse systems are increasing unions of inverse limits. It is clear that direct limits of principal groupoids (groupoids arsing from equivalence relations) are principal. To see that inverse limits of principal groupoids are principal, we just need to recall that a groupoid $G$ is principal if and only if for every $u \in G^{(0)}$, the isotropy subgroup $\{g \in G: r(g)=s(g)=u\}$ is trivial. For  if $g$ (in the inverse limit $G$) is in its isotropy group at $u\in \G0$,  $p_{\alpha}(g)$ is in the isotropy group of $p_\alpha(u)$ for all $\alpha$. It follows that $g$ is a string of identity elements and is hence an identity element.

To see $G$ is \'etale, just notice that the object space of $G$ is exactly the union of the preimages of the object spaces of all the $G_\alpha$ via the projection maps from inverse limit. It follows then that $G$ has open object space and has a Haar system of measures and so by the characterization in section 2, we have that $G$ is \'etale.
\end{proof}

Using the contravariant functor described in Theorem \ref{functor}, we have the following easy observation.

\begin{Observation}
Every generalized inverse system $\{G_\alpha, p^\alpha_\beta,\Delta\}$ induces a directed system of groupoid convolution algebras. It is straightforward to check that if $G$ is the inverse limit $\varprojlim_\alpha G_\alpha$ then  $\bigcup_\alpha C_c(G_\alpha)$  is dense in $C_c(G)$, with respect to the inductive limit topology. 
\end{Observation}

\begin{Theorem}\label{continuousfunctor}
Let $\{G_\alpha:p^\alpha_\beta\}$ be an inverse system in $\mathcal{G}$ with surjective partial bonding maps and suppose $G = \varprojlim_\alpha G_\alpha$.. We have that $\{C^*(G_\alpha),(p^\alpha_\beta)^*\}$ and $\{C_r^*(G_\alpha),(p^\alpha_\beta)^*\}$ are inductive systems of $C^*$-algebras and, furthermore, $C^*(G) = \varinjlim_\alpha C^*(G_\alpha)$ and $C_r^*(G) = \varinjlim_\alpha C_r^*(G_\alpha)$.
\end{Theorem}
\begin{proof}
This follows immediately from the fact that $\{C_c(G_\alpha):(p^\alpha_\beta)^*\}$ is a directed system of topological $*$-algebras and that $\varinjlim_\alpha C_c(G)$ is dense in $C_c(G)$ (dense with respect to the inductive limit topology) and then by applying Propositions   \ref{inducedmorphism} and \ref{reduced}, respectively. 
\end{proof}

\subsection{Twisted Groupoid \texorpdfstring{$C^*$-Algebras}{C*-Algebras}}\label{dothetwist}

For completeness, we will show how to extend all the results of this section so far to the case of twisted groupoid $C^*$-algebras.

\begin{Definition}\label{defcocycle}
 A 2-cocycle is a map $\sigma : \G2 \to \TT$ such that whenever $(g, h), (h, k) \in \G2$ we have
$$\sigma(g, h)\sigma(gh, k) = \sigma(g, hk)\sigma(h, k).$$
\end{Definition}

The map $\sigma(g, h) = 1$ for all $(g, h) \in \G2$ is always a 2-cocycle, called the \textbf{trivial cocycle}; the set of 2-cocycles on $G$ form a group under pointwise multiplication and pointwise inverse. 

\begin{Definition}
If $(G,\sigma_G)$ and  $(H,\sigma_H)$ are  locally compact groupoids with $2$-cocycles,  a morphism $q:(G,\sigma_G) \to (H, \sigma_H)$ is said to be \textbf{cocycle preserving} if it is a proper morphism of groupoids such that $\sigma_G(g,h) = \sigma_H(q(g),q(h))$. 
\end{Definition}

If $(G, \sigma)$ is a topological groupoid with Haar system $\{\mu^u:u\in \G0\}$, we let $C_c(G,\sigma)$ denote the collection of compactly supported continuous complex valued  functions on $G$. With the following multiplication, adjoint operation, and norm $\|\cdot\|_I$, $C_c(G,\sigma)$ becomes a topological *-algebra. 

$$ f*g(x) = \int_{G} f(xy)g(y^{-1})\sigma(xy,y^{-1})\,d\mu^{s(x)}(y) $$ 

$$ f^*(x) = \overline{f(x^{-1})\sigma(x,x^{-1})}$$

$$\|f\|_I = \text{max} \left\{\sup_{x \in \G0} \int_{G} |f(g)|\,d\mu^u(g), \sup_{x \in \G0} \int_{G} |f(g^{-1})|\,d\mu^u(g)\right\}.$$

The \textbf{maximal (or full) twisted groupoid $C^*$-algebra of $G$}, denoted by $C^*(G,\sigma)$, is defined to be the completion of $C_c(G, \sigma)$ with the following norm
$$\|f\|_{\text{max}} = \sup_{\pi} \|\pi(f)\|,$$
where $\pi$ runs over all continuous (with respect to the norm $\|\cdot\|_I$) *-representations of $C_c(G, \sigma)$. 

The representations $\text{Ind}(\delta_u)$ also make sense for twisted convolution algebras. Let $u\in \G0$ and let $\theta_u$ be the measure on $G$ as defined previously. Define $\text{Ind}(\delta_u):C_c(G,\sigma) \to B(L^2(G,\theta_u))$ by 

$$ \text{Ind}(\delta_u)(f)(\xi)(x) = \int_{G} f(xy)\xi(y^{-1})\sigma(xy,y^{-1})\,d\theta_u(y).$$ 

As for non-twisted convolution algebras, we define the reduced norm $||\cdot||_r$ on $C_c(G,\sigma)$ by 

$$||f||_r = \sup_{u\in \G0} ||\text{Ind}(\delta_u)(f)||$$. The completion of $C_c(G,\sigma)$ with $||\cdot||_r$ will be called the \textbf{minimal (or reduced) twisted groupoid $C^*$-algebra} and will be denoted by $C_r^*(G,\sigma)$.

It is a straightforward to adapt the proofs of Propositions \ref{reduced}  and \ref{inducedmorphism} to show that if $q:(G,\sigma_G) \to (G,\sigma_H)$ is Haar system preserving and cocycle preserving then the induced map $q^*:C_c(H,\sigma_H) \to C_c(G,\sigma_G)$ (the pullback map)  extends to maximal and minimal completions. 

Let $(G,\sigma)$ be a topological groupoid with $2$-cocycle. Notice that if $U\subset G$ is an open subgroupoid then the restriction $\sigma|_U$ is a cocycle on $U$ and, furthermore, if $U$ is a Haar subgroupoid of $G$ then the natural inclusion $(K,\sigma|_K) \hookrightarrow (G,\sigma)$ induces a $*$-embedding $C_c(K,\sigma|_{K}) \hookrightarrow C_c(G,\sigma)$.  To avoid introducing new notation for this extended category, we will call a pair $(K,\sigma|_{K})$ a \textbf{Haar subgroupoid} of $(G,\sigma)$. For the sake of simplifying notation, if we say $(K,\tau)$ is a Haar subgroupoid of $(G,\sigma)$, we mean that $\tau = \sigma|_K$. We analogously define a \textbf{partial morphism} $(\phi,K):(G,\sigma_G)\to (H,\sigma_H)$ to be a Haar subgroupoid $(K,\tau)\subset (G,\sigma_G)$ and a cocycle preserving and Haar system preserving morphism $\phi:(U,\tau)\to (H,\sigma_H)$.  We define the \textbf{induced map} to be the composition 
$$ C_c(H,\sigma_H) \overset{f^*}\to  C_c(K,\tau) \xrightarrow{\text{canonical}} C_c(G,\sigma_G).$$

Let $\mathcal{G}_\sigma$ denote the category whose objects are pairs $(G,\sigma)$ where  $G$ is a topological groupoid (with a fixed Haar system of measures) and equipped with a 2-coclycle $\sigma:G^{(2)}\to \TT$ and whose morphisms are partial morphisms. 

\begin{Theorem}\label{twistedfunctor}
Let $\mathcal{C}$ denote the category of $C^*$-algebras with $*$-morphisms. The association $\Gamma^{\sigma}_{\text{max}}$ ($\Gamma^{\sigma}_{\text{min}}$) which takes each object of $\mathcal{G}^\sigma$ and assigns its maximal (minimal) groupoid $C^*$-algebra and assigns partial morphisms to their induced map is a contravariant functor. Moreover, both $\Gamma^{\sigma}_{\text{max}}$  and $\Gamma^{\sigma}_{\text{min}}$  are extensions of the Gelfand duality functor $\Gamma_c$.
\end{Theorem}

The reader can easily check that Theorem \ref{mainthm.limitexists} can be extended also to include cocycles. Being that we are not using cocycles in our inverse systems of groupoids for our applications, we will omit the proof and defer the reader to Section 3 of \cite{BL2} for taking inverse limits of cocycled groupoids.

\section{Examples}\label{Examples}

The purpose of this section is to show how to use Theorem \ref{mainthm.limitexists} to build a wide variety of groupoid models for inductive systems and direct limits of $C^*$-algebras.

\subsection{AF Algebras}\label{UHF Algebras}

In this subsection, we will show how to model lots of unital morphisms between finite dimensional $C^*$-algebras. It will more or less follow from Theorem \ref{mainthm.limitexists} that one can model any inductive system of AF-algebras and therefore one can model any AF algebra with a groupoid, an already well known result (see \cite{ER} for example). The key focus here is not that we can model any (unital) AF algebra, but that we can model, up to unitary equivalence, all unital $*$-embeddings between finite dimensional $C^*$-algebras.

\begin{Example}\label{matrixgroupoid}
For any $n$, there exists a groupoid $G_n$ with $M_n$ as groupoid $C^*$-algebra. 
\end{Example}
\begin{proof}
Let $G_n$ denote the product groupoid $\{1,2,\hdots, n\} \times \{1,2,\hdots, n\}$; i.e. the trivial equivalence relation on an $n$ point set.  We notice that $G_n$ has precisely $n^2$ elements and, if $G$ is given the discrete topology and Haar system of measures equal to the counting measures over each object, one can identify $C^*(G_n)$ with the set $M_n$ of $n \times n$ matrices by simply noticing that the characteristic function of a point $(i,j)$ behaves exactly as the elementary matrix $e_{i,j}$.
\end{proof}

\begin{Remark}\label{ProdOfMatrixgroupoids}
It is very easy to see by definitions that $G_n\times G_m \cong G_{mn}$ for any numbers $m$ and $n$ (in fact, this follows from the commutativity of taking products of sets). 
\end{Remark}

As we know that any finite dimensional $C^*$-algebra is isomorphic to $M_{n_1} \oplus M_{n_2} \oplus \cdots  \oplus M_{n_k}$, we have the following example, which follows from Example \ref{matrixgroupoid} and the fact that $C^*(G\sqcup H) \cong C^*(G)\oplus C^*(H)$ where the disjoint union of groupoids is defined in the canonical way.

\begin{Example}\label{finitedimmodel}
For any finite dimensional $C^*$-algebra $F$, there exists a groupoid $G_F$ whose groupoid $C^*$-algebra is $F$.
\end{Example}

A subclass of AF algebras of immense importance are the UHF algebras, which are inductive limits of matrix algebras where the connecting maps are unital. As an illustration of the \ref{mainthm.limitexists}, we describe UHF algebras through a limiting process. Our building blocks are the groupoids $G_n$ from \ref{matrixgroupoid}.

\begin{Example}\label{UHF}
For any UHF algebra $U$, there exists a groupoid $G$ whose groupoid $C^*$-algebra is $U$.
\end{Example}
\begin{proof}
Let $n = \Pi_{k=0}^{\infty}\,n_k$ be the supernatural number associated with the sequence of natural numbers $n_1,n_2, \dots $, and let $U$ be the UHF algebra associated with $n$. Defining $G^k=G_{n_1} \times G_{n_2} \times \hdots \times G_{n_{k}}$, we consider the generalized inverse system with bonding maps given by  the projection $ G_{n_1} \times G_{n_2} \times \hdots \times G_{n_{l-1}}\times G_{n_l}^0 \to G^{l-1}$, which by Proposition \ref{reasonforgenpullbacks} induces an unital embedding given by $T\to  T\otimes \text{id}_{l}$. By   \ref{mainthm.limitexists} we conclude that $G_n = \varprojlim_{k} G^k$ exists, and it clearly has $C^*$ algebra $U$.

\end{proof}

A certain class of the above groupoid models for UHF-algebras have nice self absorbing properties. Recall that Proposition \ref{producttotensor} says that the groupoid $C^*$ algebra of a product $G\times H$ is the tensor product $C^*(G)\otimes C^*(H)$. It follows then that if a groupoid $G$ has the property that if $G\times G$ and $G$ are isomorphic (via a Haar system preserving isomorphism), then $C^*(G)$ is self absorbing. We say that a groupoid $G$ is \textbf{self absorbing} if there is a Haar system preserving isomorphism betweeen $G$ and $G\times G$. Notice that if $G$ is self absorbing, then one clearly has $C_c(G)\otimes C_c(G) \cong C_c(G)$ and hence self absorbing groupoids are stronger than self absorbing $C^*$-algebras as there is no need to appeal to completions.

\begin{Proposition}
Let $n$ be a supernatural number such that, if a prime $p| n$ then $p^{\infty} |n$, then the UHF groupoid $G_{n}$ is self absorbing.
\end{Proposition}
\begin{proof}
Let $n = \Pi_{i} p_i^\infty$ and, for each $k\ge 1$, let $G^k = \Pi_{i=1}^k G_{p_i^{2k}}$ and let the bonding maps $G^k \to G^{k-1}$ be defined as follows: notice that $G^k =G^{k-1} \times \Pi_{i=1}^{k-1} G_{p_i^{2}} \times G_{p_k^{2k}}$ and notice that, by Proposition \ref{reasonforgenpullbacks}, the partial bonding map $G^{k-1} \times \Pi_{i=1}^{k-1} \G0_{p_i^{2}} \times \G0_{p_k^{2k}} \to G^{k-1}$ given by coordinate projection induces the morphism $T\to T\otimes \text{id}_{p_1^2p_2^2\hdots p_{k-1}^2p_k^{2k}}$. Clearly, the inverse limit of the sequence of $G^k$ with bonding maps just described is a groupoid model for $M_n$. Alternatively, suppose that $H^k$ was defined to be $\Pi_{i=1}^k G_{p_i^{k}}$ with bonding maps defined analogously, we would also have that the inverse limit of the groupoids $H^k$ would be a groupoid model for $M_n$. Notice that the generalized inverse system $\{G^k\}$ is equal to a subsystem of $\{H^k\}$ (as $G^k = H^{2k}$ and the bonding maps match up exactly) so they clearly have the same generalized inverse limit, call it $G$. 

Notice that $G\times G = \varprojlim (H^k\times H^k) = \varprojlim G^k = G$ and so we conclude that $G$ is self absorbing.

\end{proof}

\begin{Example}\label{AF}
For every unital AF-algebra $A$, there exists a generalized inverse sequence $\{G_\alpha,p^\alpha_\beta,D\}$ where
\begin{itemize}
    \item $G_\alpha$ is a model for a finite dimensional algebra as constructed in Example \ref{finitedimmodel} for each $\alpha$.
    \item $\varprojlim G_\alpha$ is a groupoid model for $A$.
\end{itemize}
\end{Example}
\begin{proof}
Recall that a unital AF algebra is a direct limit of finite dimensional algebras with unital bonding maps. As soon as we have shown that any unital bonding map between finite dimensional algebras can be modeled with partial morphisms, we are finished because one can use Example \ref{finitedimmodel} for the pieces in the direct system. Note also, that our problem can be reduced to showing that the unital embeddings $M_k\hookrightarrow \sum_{j} M_{jk}$ given by $T\to \sum_{j } T\otimes \text{id}_{j}$ can be obtained as an induced map of a partial morphism. This just amounts to using Proposition \ref{reasonforgenpullbacks} on each piece of the groupoid $\bigsqcup_{j} G_{jk}$.
\end{proof}

It would be interesting to understand the relationship between the groupoid constructed in the proof of Theorem \ref{AF} and the groupoids obtained in \cite{ER} obtained using tail equivalence relations on a Bratteli Diagram.

\subsection{Infinite Tensor Powers of Groupoid \texorpdfstring{$C^*$}{C*}-algebras}\label{infinitetensorpowers}

In this section, we will provide an alternative description of infinite tensor powers of groupoid $C^*$-algebras using our inverse limit machinery. If the reader wishes only to deal with unital $C^*$-algebras, then they just need to recall that a groupoid $C^*$-algebra is unital when the groupoid in question is \'etale with compact object space.

First, we will describe a ``spectral decomposition'' of the $C^*$-algebra $\bigotimes_{i=1}^{\infty}(F) := \varinjlim_j \otimes_{i=1}^j F$ where $F$ is a finite dimensional $C^*$-algebra as this algebra has many important applications in dynamics.  Let $n_1,n_2,\hdots,n_k \in \ZZ_{\ge 0}$ be such that $F =\oplus_{i=1}^k M_{n_i}$.  It is easy to see that this is isomorphic to the continuous sections of a continuous field of UHF algebras over the Cantor set $K = \{n_1,n_2,\hdots , n_k\}^{\NN}$ in which the fiber of $(n_{\alpha(i)})_{i\in \ZZ }$ is the algebra $\otimes_{i\in \ZZ} M_{n_{\alpha_i}}$. Notice that each of the fibers of the bundle are unital and hence $C(K)$ embeds into $A$. In fact, $C(K)$ is exactly the center of $A$. 

Let $G = G_{n_1} \sqcup G_{n_2}\sqcup \hdots \sqcup G_{n_k}$ with counting measures as Haar system. The groupoid $C^*$-algebra for $G$ is clearly $\oplus_{i=1}^nM_{n_1} = F$. For each $k\ge 1$, we define the groupoid $G^k = \underbrace {G\times G\times \hdots G}\limits_{\text{n times}}$ and notice that counting measures is a Haar system of measures for $G^k$. Clearly the generalized inverse system $(G^k, \pi_k)$ where the partial maps $\pi_k:G^k\times \G0 \to G^k$ is the projection mapping has generalized inverse limit $H$ with $\bigotimes_{i=1}^{\infty}(F)$ as groupoid $C^*$-algebra.

In general, let $\{G_i\}$ be a countable collection of topological groupoids . By all the observations made thusfar, it is clear that the inverse limit of  the groupoids $H_k = \Pi_{i=1}^k G_k$ with partial bonding maps $\pi^m_k:H^k \times \G0_{k+1}\times \hdots \G0_{k_m} \to H_k$ for $m\ge k$ has groupoid $C^*$-algebra equal to $\otimes_{i\in \NN} C^*(G_i)$. The analogous statements hold for reduced completions.

\subsection{Groupoid Models for The Crossed product of Infinite Tensor Powers by Bernoulli Actions}

Let $\Gamma$ be a countable discrete group and let $G$ be a locally compact Hausdorff groupoid with Haar system of measures $\{\mu^x:x\in \G0\}$. We will show how the crossed product $\bigotimes_\Gamma C^*(G)\rtimes \Gamma$ has a groupoid model in the case the action is by the Bernoulli shift (i.e. $\Gamma$ just shifts the indices of the tensor factors). Let $G^\infty$ denote the groupoid model for $\bigotimes_{\Gamma} C^*(G)$ as in the previous subsection. Let $\gamma \in \Gamma$ and notice that the shift of the indices on the left by $\gamma$ of $G^\infty$ when viewed as a subspace of the product $\Pi_{\Gamma} G$ induces a Haar system preserving groupoid isomorphism of $G^\infty$ whose induced map is exactly the automorphism of $\bigotimes_\Gamma C^*(G)$ given by shifting the indices on the left by $\gamma$. The reader can easily verify that shifting the indices of $G^\infty$ on the left by elements of $\Gamma$ is an action of $\Gamma$ on $G$ and, by our observations, it induces the Bernoulli action on $\bigotimes_{\Gamma} C^*(G)$ (see \cite{Brown} for a definition of one groupoid acting on another and what the associated semi-direct product groupoid is). It is a straightforward check that the groupoid $C^*$-algebra of the semi-direct product groupoid $G^\infty \rtimes \Gamma$ (whose Haar system is given by the product the original Haar system with counting measure) is a groupoid model for $\bigotimes_\Gamma C^*(G) \rtimes \Gamma$. By the same reasoning, we also have that  $C^*_r(G^\infty\rtimes \Gamma) \cong \bigotimes_{\Gamma}C^*_r(G)\rtimes \Gamma$. 

The importance of examples like this is the connection to the group $C^*$-algebras of wreath products like $\ZZ \wr \ZZ$, see \cite{DaPS}.

\subsection{Groupoid Crossed Products By an Endomorphism}

Recall that the Cuntz Algebras $O_n$ for $1\leq n <\infty$ are defined as the crossed product of the UHF algebra $M_{n^\infty}$ by the cannonical endomorphism which maps $M_{n^{\infty}}\otimes M_n$ isomorphically into its subalgebra $M_{n^\infty}\otimes e_{1,1}$.

There are many groupoid models for $O_n$, but we  want to explain how one can create an alternative groupoid model for $O_n$ in a way that highlights our above description. First, if $G$ is the groupoid model of $M_{n^{\infty}}$ obtained in Example \ref{UHF} then notice that $G\times G_n \cong G$ (where $G_n$ is defined in as Example \ref{matrixgroupoid}), the latter which is clearly groupoid isomorphic (preserving the Haar system) to the Haar subgroupoid $G\times \{e_{1,1}\}$ of $G\times G_n$. Let $\phi: G\times G_n \to G$ be the partial map which projects the Haar subgroupoid $G\times \{e_{1,1}\}$ to its first coordinate. The reader can easily check that the induced map is the endomorphism described above.

The only examples of groupoid crossed products by endomorphisms that seem to exist in the literature is the work \cite{D} by V. Deaconu. In his work, he shows how to define a crossed product of a compact Hausdorff space by a surjective self map and shows that his construction corresponds to crossed products when the self map is a genuine homeomorphism.  We propose an extension of Deaconu's construction by considering surjective partial maps of groupoids. If the reader is interested in such a construction, we refer the reader to the work \cite{Brown} by R. Brown for an introduction to groupoid semi-direct products.

\section{Basics of Modeling with Partial Morphisms}\label{basicsofmodeling}

One of the easiest ways of constructing new groupoids from old ones will be to use the quotient criterion which shows up as Proposition 3.4 of \cite{AG}.

\begin{Proposition}\label{Haarquotient}[Quotient Criterion]
Let $G$ and $H$ be topological groupoids and let $q:G\to H$ be a surjective morphism of  groupoids such that $q$ is topologically a quotient map (or an open map). Suppose $G$ has a Haar system of measures $\{\mu^x:x\in G^0\}$. If, for all $u\in H^{0}$ and for all $v\in q^{-1}(u)$ and all $f\in C_0(H)$, we have 

 $$\int_G (f \circ q) \, d\mu^u = \int_G (f \circ q) \, d\mu^v$$
then $H$ admits a natural Haar system of measures $\{\nu^u:u\in H^{0}\}$ that makes $q$  Haar system preserving.
\end{Proposition}

\begin{Definition}
Suppose that $G$ is a locally compact groupoid with fixed Haar system and Haar subgroupoid $H$. If $q:H\to K$ is a quotient map that satisfies the hypotheses of Proposition \ref{Haarquotient}, then we call $K$, with the induced Haar system, a \textbf{partial quotient of $G$} and we call $q$ the \textbf{partial quotient map}. 
\end{Definition}

The following Lemma follows easily from Lemma \ref{producttotensor} and Example \ref{matrixgroupoid}. 

\begin{Lemma}
For any natural number $n$ and locally compact space $X$, there exists a groupoid $G_n(X)$ whose groupoid $C^*$-algebra is $M_n(C_0(X))$.
\end{Lemma}

For a natural number $n$, let $Z_n$ denote the $C^*$-subalgebra of $C([0,1],M_n)$ of functions $f$ such that $f(0) = I$.

\begin{figure}
    \centering
    \includegraphics[width=12cm]{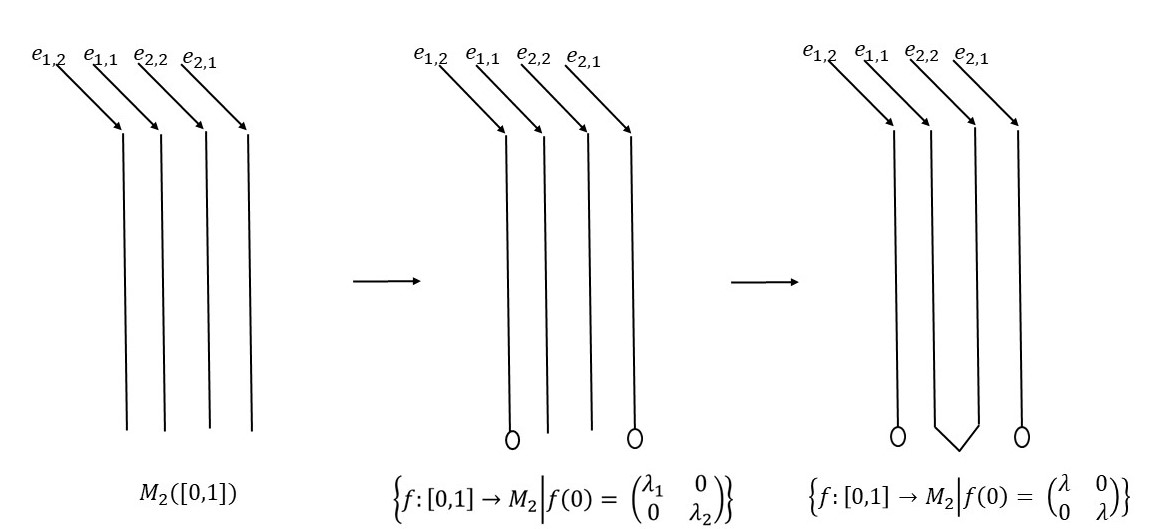}
    \caption{}
\end{figure}

\begin{Lemma}\label{G_{Z_n}}
For each natural number $n$, there exists a groupoid $G_{Z_n}$ whose groupoid $C^*$-algebra is $Z_n$.
\end{Lemma}
\begin{proof}
Let $G_{Z_n}' = G_n\times [0,1] \setminus \{(g,0): g\in \G1\setminus \G0\}$. It is clear that $G_{Z_n}'$ is a proper \'etale subgroupoid of $G_n\times [0,1]$ as it is $r$-discrete and the range map is clearly open since the range map for $G_n\times [0,1]$ is open. Notice that the inclusion of $G_{Z_n}'\hookrightarrow G_n\times [0,1]$ is the inclusion of all the  functions $f:[0,1]\to M_n$ where 

\begin{center}
$f(0) = 
\begin{pmatrix} 
\lambda_1 & 0 & 0 & \hdots & 0 \\
0 & \lambda_2 & 0 & \hdots & 0 \\
0 & 0 & \lambda_3 & \hdots & 0 \\
\vdots & \vdots & \vdots & \ddots & 0\\
0 & 0& 0 & \hdots & \lambda_n
\end{pmatrix}
$     
\end{center}

Let $G_{Z_n} = G_{Z_n}'/\sim$ where $(g,0)\sim (h,0)$ for all $g,h\in \G0$. Notice that the quotient map  $q:G_{Z_n}'\to G_{Z_n}$ is Haar system preserving and that the image of the induced map $q^*$ is exactly all the functions $f:[0,1]\to M_n$ where $f(0) = \lambda I$ where $I$ denotes the identity of $M_n$ (see Figure 1 for a geometric visualization when $n=2$).
It follows that the groupoid $C^*$-algebra of $G_{Z_n}$ is $Z_n$ and the induced map of the parital map $q$ is the inclusion of those functions $f\interval \to M_n$ such that $f(0) = \lambda I$ for some $\lambda\in \CC$.
\end{proof}

\begin{ignore}
{
Recall the definition of supernatural numbers as formal infinite products of powers ($0$, natural numbers or $\infty$) of primes or $1$'s, and note the  definition of the dimension drop algebras $Z_{m,n}$ (see \cite{JS,RW}).
}
\end{ignore}

\subsection{Dimension Drop Groupoids}

\begin{Definition}
If $m$ and $n$ are  natural  numbers, define the \textbf{dimension drop algebra} $Z_{m,n}$ by 
\begin{equation*}
    Z_{m,n}=\{f \in C([0,1], M_m \otimes M_n) : f(0) \in M_m \otimes \mathbb{C}, f(1) \in \mathbb{C} \otimes M_n \}
\end{equation*}

\end{Definition}

\begin{Lemma}\label{dimensiondropmodel}
For each pair of  natural numbers $m$ and $n$, there exists a groupoid, $G_{m,n}$ whose groupoid $C^*$-algebra is $Z_{m,n}$. 
\end{Lemma}
\begin{proof}
The construction of $G_{m,n}$ is almost identical to the construction of $G_{Z_n}$ as in the proof of Lemma \ref{G_{Z_n}}.

Let $G_{m,n}'$ be the open subspace of the groupoid $G_m \times G_n \times [0,1]$ taken by deleting the points $\{(h,g,0): g\in \G1_n\setminus \G0_n\}$ and $\{(h,g,1): h\in \G1_m\setminus \G0_m\}$. It is clear that the groupoid $C^*$-algebra of $G_{m,n}'$ embeds into the groupoid $C^*$-algebra of $G_m \times G_n \times [0,1]$ (equal to $C(\interval,M_m\otimes M_n)$) as all the functions $f:\interval\to M_m\otimes M_n$ which satisfy $f(0) = D_0\otimes M_0$ and $f(1) = M_1\otimes D_1$ where $D_i$ are diagonal matrices and $M_i$ are any matrices.
 
Now, let $q:G_{m,n}'\to G_{m,n}$ be the quotient under the equivalence relation $(g,h,1)\sim (g,h',1)$ and $(g,h,0)\sim (g',h,0)$. Once can check that the image of the induced map $q^*$ in $C^*(G_m \times G_n \times [0,1]) = C(\interval,M_m\otimes M_n)$ is exactly the dimension drop algebra $Z_{m,n}$.
\end{proof}

\subsection{Twisting Intervals of Matrix Groupoids by Paths of Unitaries}\label{addthetwist}

We will be interested in taking the groupoid $G_n(\interval)$ (and its subgroupoids and quotients thereof) and considering a twist by a path $u_t$ of unitaries in $M_{n}$. The issue is that we cannot do the twisting in $G_n(\interval)$. We need a larger ambient groupoid to make the twist. Let $\{e_{i,j}: 1\leq i,j\leq n\}$ be the usual orthonormal basis for $M_n$. For each unitary $U\in U_n$, notice that $G_n^U := \{U^*e_{i,j}U: 1\leq i,j\leq n\}$ carries a natural \'etale groupoid structure and is isomorphic to $G_n$. Let $V_n$ be the isomorphic (as topological groupoids) copy of $U_n\times G_n$ given by $\{(U,x): x\in G_n^U \text{ and } U\in U_n\}$. Recall that the source and range maps for a groupoid  $G$ to its object space $\G0$ are denoted by $s$ and $r$, respectively. We define $s(U,x) = (U,s(x))$ and similarly for $r$. We will call $V_n$ the \textbf{groupoid Steifel variety} as it is a groupoid indexed over orthornmal frames of $\CC^n$. Observe that $G_n$ embeds into $V_n$ as the set of tuples $\{(\text{id}_n,e_{i,j}): 1\leq i,j \leq n\}$.

Notice that a unitary $U\in U_n$ defines a groupoid isomorphism $\hat{U}$ between the subgroupoids  $\{(V,x): x\in G^V_n\}$ and $\{(UV,x): x\in G_n^{UV}\}$ defined by $(V,x) \to (UV, U^*xU)$. $\hat{U}$ is clearly Haar system preserving and induces a groupoid $C^*$-algebra isomorphism. 

 Notice that $G_n\times \interval$ is an \'etale subgroupoid of the product groupoid $V_n\times \interval$. Notice also that any path of unitaries $u_t$ define a topological groupoid isomorphism of $G_n\times \interval$ via the function $\widehat{u_t}$ defined by  $\widehat{u_t}((U,x)) = (Uu_t, u_t^*xu_t)$ and hence the image of $G_n\times \interval$, which we will denote by $G^u_n(\interval)$, under $\widehat{u_t}$ is also a second countable locally compact and Hausdorff \'etale groupoid and is the collection given by $\{(u_t,u_te_{i,j}u_t^*,t):t\in \interval \text{ and } 1\leq i,j \leq n\}$.  Notice then that each of the groupoids $G^u_n(\interval)$ (one for each such path $u_t$)  are isomorphic by an isomorphism that is a homeomorphism and preserves Haar systems. We will refer to groupoids $G^u_n(\interval)$ as \textbf{paths of the groupoids $M_n$}. 
 
 \begin{ignore}
{
 We will not need to work with anything more general than paths of groupoids, but it is worth noticing that, for any locally compact Hausdorff space $X$ and any proper continuous function $u:X\to U_n$, one can define $G_n^u(X)$ as the subspace of the product $G_n\times X$ given by $\{(u_x,u_xe_{i,j}u_x^*,x):x\in X \text{ and } 1\leq i,j \leq n\}$. We will call $G^u_n(X)$ a \textbf{u-twisted model for $C_0(X,M_n)$}.

What we will be most interested in are twists of the groupoids $G_{m,n}$ and $G(n,n')$ by paths of unitaries. Naturally, one must be careful because these groupoids are quotients of open subgroupoids of $G_k(\interval)$ for an appropriate choice of $k\ge 1$. Let's start with explaining twists of $G_{m,n}$. Let $u_t$ be a path of unitaries in $M_{mn}$ as above. $G_{m,n}$ is the quotient of the open subgroupoid $G_{m,n}'$ of $G_m\times G_n \times \interval$ given by removing the points $\{(g,h,0): g\in G\setminus \G0\}\cup \{(g,h,1): h\in H\setminus H^{(0)}\}$. Let $G_{m,n}''$ be the image of $G_{m,n}'$ in $G^u_{mn}(\interval)$ under the twisting of $G_n\times G_m\times \interval$ by the path of unitaries $u_t$. From there, we make identifications depending on the unitaries $u_0$ and $u_1$. We define $\sim$ on $G_{m,n}'$ to be the image of the equivalence relation on $G_{m,n}(\interval)$ and define $G_{m,n}^u$ to be the resulting quotient. It is straightforward to check that the groupoid $C^*$-algebra of $G_{m,n}^u$ is isomorphic to the dimension drop algebras $Z_{m,n}$. Performing the same procedure, we can construct groupoids $G^u(n,n')$. 
}
\end{ignore}
\subsection{Standard Subalgebras}\label{generallyspeaking} Recall (by Lemma 3.10 of \cite{KV} for instance) that every unital subalgebra of $M_n$ is conjugate (by some invertible matrix) to an algebra of the form $\oplus_{i=1}^k \text{id}_{a_i}\otimes M_{b_i}$ where $\sum_{i=1}^k a_ib_i = n.$ Using projections, one can easily show that in fact every subalgebra with a unit (perhaps not equal to the unit on $M_n$) is of the form $\left(\oplus_{i=1}^k \text{id}_{a_i}\otimes M_{b_i}\right) + 0_m$ where $\sum_{i=1}^k a_ib_i = n - m$ (where $0_m$ denotes the $0$ matrix on $\CC^m$).  We will call subalgebras of that form to be \textbf{standard subalgebras}.


\begin{Proposition}\label{cutandpaste}
Let $n\ge 1$, $u:\interval\to U_n$ a proper continuous map, and $F\subset \interval$ be a finite subset and let $A_x \subset M_n$ be a standard $C^*$-subalgebra for every $x\in F$. There exists a groupoid $G(A_x:x\in F)$ which is a subquotient of $G^u_n(\interval)$ (defined in Subsection \ref{addthetwist})  whose $C^*$-algebra is isomorphic to the $C^*$-subalgebra $C$ of $u_tC(\interval,M_n)u_t^*$ (at least, the isomorphic copy under the induced map) which consists of those functions $f$ such that $f(x)\in A_x$. 

Moreover, if $B_x\subset A_x$ and $B_x$ is a standard  subalgebra for each $x\in F$, then the groupoid model $G(B_x:x\in F)$ is a subquotient of $G(A_x:x\in F)$.  
\end{Proposition}
\begin{proof}
Assume that the path $u_t$ is the constant path equal to the identity for all $t$. The general case is a minor adaptation of this case.

Let $G_n$ be defined as in Example \ref{matrixgroupoid}. We will outline how to construct a groupoid $H$ which is a quotient of an open Haar subgroupoid of $G_n$ (as defined in the proof of Lemma \ref{matrixgroupoid}) such that the embedded image of $C^*(H)$ in $C^*(G)$ under the induced map is exactly of the form $\oplus_{i=1}^k \text{id}_{a_i}\otimes M_{b_i}$. Notice that the elements of $G_n$ correspond to matrix entries by the identification of the characteristic function $\chi_{(i,j)}$. Notice that the subalgebra  $\oplus_{i=1}^k \left(\oplus_{j=1}^{a_i}M_{b_i}\right)$ of $M_n$ corresponds uniquely to an open subgroupoid $H'$ of $G_n$ simply by removing the points of $G_n$ that correspond to matrix entries that are $0$. If we label the subset of arrows that corresponds to the subalgebra $\oplus_{j=1}^{a_i}M_{b_i}$ by $\{e_{i,j}^k: 1\leq i,j \leq b_i \text{ and } 1\leq k \leq a_i\}$ and, for each $i$ make the identifications $e_{i,j}^{k_1} \sim e_{i,j}^{k_2}$ for all $1\leq k_1,k_2\leq a_i$ then it is easy to see that the quotient map $q: H' \to H := H'\sim$ satisfies the hypothesis of Proposition \ref{Haarquotient} and hence we can give the quotient $H$ a Haar system of measures that makes the quotient map Haar system preserving (in this case, it is just counting measure over each fiber). It is a straightforward check that the  embedded image of the induced map from $G_n$ to $H$ is the inclusion of the subalgebra $\oplus_{i=1}^k \text{id}_{a_i}\otimes M_{b_i}$.

Perform this procedure for $(x,G_n)$ for each point $x\in F$ and notice that the induced map of the resulting subquotient will be exactly the inclusion of the subalgebra $C$. We leave the simple proof of the moreover assertion to the reader.
\end{proof}

\section{Groupoid Model for the Jiang-Su Algebra and  the Razak-Jacelon Algebra}\label{mainsection}

\begin{Definition}
[Jiang-Su \cite{JS}]\label{Defofjiangsu}
The \textbf{Jiang-Su algebra}, denoted by $\mathcal{Z}$, is the inductive limit of any sequence $A_1 \xrightarrow{\phi_1} A_2 \xrightarrow{\phi_2} A_3 \xrightarrow{\phi_3}  \hdots $, where $A_n = Z_{p_n,q_n}$ is a prime dimension drop algebra such that the connecting morphisms $\phi_{m,n} = \phi_{n-1} \circ \hdots \circ \phi_{m+1} \circ \phi_m:A_m\to A_n $ is an injective morphism of the form 

\begin{center}
$   \phi_{m,n}(f) = 
u^*\begin{pmatrix} 
f\circ \xi_1 & 0 & 0 & \hdots & 0 \\
0 & f\circ \xi_2 & 0 & \hdots & 0 \\
0 & 0 & f\circ \xi_3 & \hdots & 0 \\
\vdots & \vdots & \vdots & \ddots & 0\\
0 & 0& 0 & \hdots & f \circ \xi_k
\end{pmatrix}u
$
\end{center}
where $f\in A_m$, $u$ is a continuous path of unitaries in $U_{M_{p_nq_n}}$, $\xi_i$ is a sequence of continuous paths in $\interval$, each satisfying 

$$|\xi_i(x) -\xi_i(y)| \leq \frac{1}{2^{n-m}} \hspace{1cm} \forall x,y\in \interval $$

and 
$$\bigcup_{i=1}^k \xi_i(\interval) = \interval$$
\end{Definition}

\begin{Definition}\label{generaldimensiondropmodels}
For every pair of natural numbers $p$ and $q$ and path $u_t:\interval \to U_n$, let $G^u_{p,q}$ denote the groupoid obtained using Proposition \ref{cutandpaste} which is a subquotient of $G_{pq}^u(\interval)$ whose $C^*$-algebra is $Z_{p,q}$. For the constant path $u_t:\interval \to \{\text{id}_n\}$, we denote the groupoid $G^u_{p,q}$ by $G_{p,q}$.
\end{Definition}

\begin{Theorem}\label{JiangSu}
 There exists a  generalized inverse sequence $\{H_i,\phi^i_j\}$ of groupoids whose inverse limit groupoid $G_\mathcal{Z}$ is a groupoid model for the Jiang-Su algebra. Moreover, $G_{\mathcal{Z}}$ is an \'etale equivalence relation on a compact metric space with covering dimension at most 1.
\end{Theorem}

\begin{Remark}
In fact, the dual directed system of $C^*$-algebras is equal to the system constructed in the proof of Proposition 2.5 in \cite{JS}. 
\end{Remark}

\begin{proof}
The main idea, in view of Section \ref{maintool}, is to show that the bonding maps created  in the proof of Proposition 2.5 of \cite{JS} are induced maps of partial morphisms defined between certain groupoid models for dimension drop algebras. As soon as we have modeled the bonding maps as partial morphisms, one just applies \ref{mainthm.limitexists} to obtain a groupoid model for the Jiang-Su algebra.

Throughout the proof, we will use $V_{k}$ to denote the groupoid Steifel variety as defined in Section \ref{addthetwist} and we will use $G_k$ to denote the groupoid model of $M_k$ as defined in the proof of Example \ref{matrixgroupoid}. 

As usual, we will define $H_i$ inductively. Let $H_1 = G_{2,3}$ as defined in the proof of Lemma \ref{dimensiondropmodel} 
and suppose that $H_i$ and the bonding maps $\phi^i_j$ have been chosen for $j\leq i$  so that 
\begin{enumerate}
    \item $H_i$ is equal to $G^U_{p_{i},q_{i}}$ for some path  of unitaries $U_t$ in $M_{p_{i},q_{i}}$ (see Definition \ref{generalizedpullback}).
    \item the induced map of $\phi^i_j$, denoted by $\phi_{j,i}$, satisfies all the required properties as laid out in Definition \ref{Defofjiangsu}. 
\end{enumerate}

\underline{The construction of $H_{i+1}$ and $\phi^{i+1}_i$:} Let $k_0> 2q_i$ and $k_1> 2p_i$ be primes such that $p_ik_0$ and $q_ik_1$ are relatively prime and define $p_{i+1} = p_ik_0$ and $q_{i+1} = p_ik_1$. Let $k = k_0k_1$ and let $r_0$ be the integer such that  
$$0 < r_0 \leq q_{i+1} \hspace{.5cm} \text{ and } \hspace{.5cm} r_0 \equiv k (mod\hspace{.1cm} q_{i+1}). $$ 

As noted in the proof of Proposition 2.5 in \cite{JS}, we have that $r_0q_i$ and $k-r_0$ are both divisible $q_{i+1}$.

Similarly, define $r_1$ to be the unique integer such that
 
$$0 < r_1 \leq p_{i+1} \hspace{.5cm} \text{ and } \hspace{.5cm} r_1 \equiv k (mod \hspace{.1cm} p_{i+1}). $$ 

Again, we have  $k-r_1$ and $r_1p_i$ are divisible by $p_{i+1}$. As observed in \cite{JS} (at the bottom of page 370), we have that $k-r_1-r_0 >0$.

In order to construct $H_{i+1}$, we must first appeal to an auxiliary groupoid which we will denote by $K$. We construct $K$ as follows: 

Define embeddings $\iota_j:M_{p_iq_i}\hookrightarrow M_k(M_{p_iq_i})$ for $j=1,2, \hdots k_0k_1$ where $\iota_j$ is the defined by

\begin{center}
$  \iota_j(T)\hspace{.1cm} = \hspace{.1cm} 
\bordermatrix{~ & 1 & 2 & \hdots & j & \hdots & k  \cr
              1& 0 & 0 & \hdots & 0 & \hdots & 0 \cr
              2 & 0 & 0 & \hdots & 0 & \hdots & 0 \cr
              \vdots & \vdots  & \vdots & \vdots & \vdots & \vdots & \vdots \cr
              j & 0 & 0 & \hdots & T  & \hdots  &0  \cr
              \vdots & \vdots  & \vdots & \vdots & \vdots & \vdots & \vdots \cr
              k & 0  & 0 & 0 & 0 & 0 & 0 \cr}
$
\end{center}

Next define a path $W_t$ of unitaries in $M_{p_{i+1}q_{i+1}}$ as follows. Let $\eta_j$ be some element of the standard basis element of $\CC^{p_{i+1}q_{i+1}}$ and choose $n$ to be the unique integer such that $np_iq_i \le j < (n+1)p_iq_i$. We first define $k$ paths $\xi_i:\interval\to \interval$ by

\begin{align*} 
   \xi_i(t) =   \begin{cases} 
      t/2 & i \leq r_0  \text{ and } \\ 
      1/2 & r_0 < i <r_1 \text{ and } \\       
      (t+1)/2 & r_1 \leq i \leq k  \\
   \end{cases}
\end{align*}

and then we define a path of unitaries $W_t$ by $$W_t(\eta_j) = \iota_{n}(U_{\xi_n(t)})(\eta_j).$$

Let $K$ denote the subquotient of $G_{p_{i_1}q_{i+1}}^W(\interval)$ such that $C^*(K)$ consists of those functions $f:\interval \to \oplus_{l=1}^k M_{p_iq_i}$ such that 


\begin{center}
$  f(0) \hspace{.1cm} = \hspace{.1cm} 
\bordermatrix{~ & 1 & 2 & \hdots & r_0 & r_0+ 1& \hdots & k  \cr
              1&  \text{id}_{q_i} \otimes M & 0 & \hdots & 0 & 0 & \hdots & 0 \cr
              2 & 0 & \text{id}_{q_i}\otimes M & \hdots & 0 & 0&  \hdots & 0 \cr
              \vdots & \vdots  & \vdots & \vdots & \vdots & \vdots & \ddots & \vdots \cr
              r_0 & 0 & 0 & \hdots & \text{id}_{q_i}\otimes  M  & 0&\hdots  &0  \cr
              r_0+1 & 0 & 0 & \hdots & 0 & M' & \hdots  &0  \cr
              \vdots & \vdots  & \vdots & \vdots & \vdots & \vdots & \ddots & \vdots \cr
              k & 0  & 0 & 0 & 0 & 0 & 0 & M' \cr}
$
\end{center}

and 

\begin{center}
$  f(1) \hspace{.1cm} = \hspace{.1cm} 
\bordermatrix{~ & 1 & 2 & \hdots & k-(r_1 +1) & k-r_1& \hdots & k  \cr
              1& N' & 0 & \hdots & 0 & 0 & \hdots & 0 \cr
              2 & 0 & N' & \hdots & 0 & 0&  \hdots & 0 \cr
              \vdots & \vdots  & \vdots & \vdots & \vdots & \vdots & \ddots & \vdots \cr
              k-(r_1+1) & 0 & 0 & \hdots & N'  & 0&\hdots  &0  \cr
              k-r_1 & 0 & 0 & \hdots & 0 & \text{id}_{p_i}\otimes  N & \hdots  &0  \cr
              \vdots & \vdots  & \vdots & \vdots & \vdots & \vdots & \ddots & \vdots \cr
              k & 0  & 0 & 0 & 0 & 0 & 0 & \text{id}_{p_i} \otimes N \cr}
$
\end{center}
where $M \in M_{p_i}$, $N\in M_{q_i}$, and  $M',N' \in M_{p_iq_i}$. $K$ is the Haar system preserving quotient of the groupoid $K_0$ which is the disjoint union of $k$ groupoids where the first $r_0$ many are all of the form $L^1$, $k-r_0-r_1$  are of the form $L^2$, and $r_1$ many are of the form $L^3$ where we define $L^i$ as follows: $L^1$ is obtained by taking the closed subgroupoid of $H_i$ given by taking all coordinates with $0\leq t\leq 1/2$ and then by stretching the inverval coordinate by $2$; so that 0 is fixed and $1/2$ is pulled to 1.  $L^2$ is the groupoid $G^{U_{1/2}}_n \times \interval$ and $L^3$ is the closed subgroupoid of $H_i$ given by taking all coordinates with $1/2\leq t\leq 1$ and then by stretching the interval coordinate the function $2t-1$ (the inverse map of $(t+1)/2$) so that the 1 coordinate is fixed and the 1/2 coordinate is pulled to 0. In order to get $K$, we need to make identifications at the interval coordinate $t=0$ and $t=1$. Notice that the groupoid $C^*$-algebra of $K_0$ is the collection of functions $f:\interval \to \oplus_{l=1}^k M_{p_{i}q_{i}}$ such that

\begin{center}
$  f(0) \hspace{.1cm} = \hspace{.1cm} 
\bordermatrix{~ & 1 & 2 & \hdots & r_0 & r_0+ 1& \hdots & k  \cr
              1&  \text{id}_{q_i} \otimes M_1 & 0 & \hdots & 0 & 0 & \hdots & 0 \cr
              2 & 0 & \text{id}_{q_i}\otimes M_2 & \hdots & 0 & 0&  \hdots & 0 \cr
              \vdots & \vdots  & \vdots & \vdots & \vdots & \vdots & \ddots & \vdots \cr
              r_0 & 0 & 0 & \hdots & \text{id}_{q_i}\otimes  M_{r_0}  & 0&\hdots  &0  \cr
              r_0+1 & 0 & 0 & \hdots & 0 & M_1' & \hdots  &0  \cr
              \vdots & \vdots  & \vdots & \vdots & \vdots & \vdots & \ddots & \vdots \cr
              k & 0  & 0 & 0 & 0 & 0 & 0 & M_{k-r_0}' \cr}
$
\end{center}

and 

\begin{center}
$  f(1) \hspace{.1cm} = \hspace{.1cm} 
\bordermatrix{~ & 1 & 2 & \hdots & k-(r_1 +1) & k-r_1& \hdots & k  \cr
              1& N_1' & 0 & \hdots & 0 & 0 & \hdots & 0 \cr
              2 & 0 & N_2' & \hdots & 0 & 0&  \hdots & 0 \cr
              \vdots & \vdots  & \vdots & \vdots & \vdots & \vdots & \ddots & \vdots \cr
              k-(r_1+1) & 0 & 0 & \hdots & N_{k-(r_1+1)}'  & 0&\hdots  &0  \cr
              k-r_1 & 0 & 0 & \hdots & 0 & \text{id}_{p_i}\otimes  N_1 & \hdots  &0  \cr
              \vdots & \vdots  & \vdots & \vdots & \vdots & \vdots & \ddots & \vdots \cr
              k & 0  & 0 & 0 & 0 & 0 & 0 & \text{id}_{p_i} \otimes N_{k-r_1} \cr}
$
\end{center}
$K$ is obtained by taking the quotient by identifying the first $r_0$ many of the groupoid along their $0$-coordinates and identifying the last $k-r_0$ many groupoids along their $0$ coordinates in the obvious way. Similarly one has to identify the first $k-(r_1+1)$ groupoids along their $1$ coordinate and also identically identify the last $r_1$ many groupoids along their $1$ level.

We define $\psi:K\to H_i$ by mapping the subgroupoids $L^1$ to $H_i$ by shrinking their $\interval$ coordinate by $2$ and sending them in the obvious way to the subset of $H_i$ whose interval coordinates are between 0 and 1/2. The groupoids of the form $L^2$ are mapped to $H_i$ by first squashing their interval coordinate to the point 1/2 and then mapping them to $H_i$ by sending them to the 1/2 level in the obvious way. We similarly define the map from $L^3$ to $H_i$ by first applying the map $(t+1)/2$ to the interval coordinates (pulling 1/2 to 0 and leaving 1 fixed) and canonically mapping them over to the points in $H_i$ with interval coordinate $1/2\leq t\leq 1$. The reader can check that $\psi$ is Haar system preserving and, furthermore, that the induced map of the partial morphism $\psi:K\to H_i$ is given by

\begin{center}
$   f\longrightarrow
\begin{pmatrix} 
f\circ \xi_1 & 0 & 0 & \hdots & 0 \\
0 & f\circ \xi_2 & 0 & \hdots & 0 \\
0 & 0 & f\circ \xi_3 & \hdots & 0 \\
\vdots & \vdots & \vdots & \ddots & 0\\
0 & 0& 0 & \hdots & f \circ \xi_k
\end{pmatrix}
$
\end{center}

As noted in \cite{JS} in the proof, there exists a path of permutation unitaries $u_t$ in $M_{p_{i+1}q_{i+1}}$ such that conjugation by $u_t$ defines an embedding of $C^*(K)$ in $Z_{p_{i+1},q_{i+1}}$. It follows that $K$ is groupoid isomorphic (via a Haar preserving groupoid isomorphism) to a subquotient $K'$ of a subquotient $H_{i+1}$ of $G^{uW}_{n_{i+1}'}$ whose groupoid $C^*$-algebra is exactly $Z_{p_{i+1},q_{i+1}}$. Let $\psi:K\to K'$ be the groupoid isomorphism and let $q:H_{i+1}\to K'$ be the partial quotient map. 

Observe now that the partial morphism $\phi^{i+1}_{i}:H_{i+1}\to H_i$ given by the composition $\Phi \circ \psi\circ q$ is exactly the morphism $\phi_{i,i+1}$ defined in the proof of Proposition 2.5 in \cite{JS}. It follows that the generalized inverse system $\{H_i= ,\phi^i_j\}$ is a groupoid model for the inductive system constructed in the proof of Proposition 2.5 in \cite{JS}.

The groupoid $G_{\mathcal{Z}}$ is an \'etale equivalence relation because each groupoid $H_n$ is an \'etale equivalence relation by Proposition \ref{equivrelation}. 
 
\end{proof}

\begin{Remark}
There is another interesting construction of $\mathcal{Z}$ given in Theorem 3.4 in \cite{RW} by considering a self map $\phi$ of an infinite prime dimension drop algebra $Z_{p,q}$ which factors through $\mathcal{Z}$ as follows $Z_{p,q}\to \mathcal{Z} \to Z_{p,q}$. It would be very interesting to understand this morphism at the level of groupoids. As $\mathcal{Z}$ is the inductive limit of $Z_{p,q}$ with constant bonding map $\phi$, this would give yet another groupoid model for $\mathcal{Z}$ by considering the generalized inverse limit of the associated groupoids.
\end{Remark}

The convenience of our constructions is that one can use essentially the same type of proof as in Theorem \ref{JiangSu} to prove that there exists a groupoid with groupoid $C^*$-algebra equal to the Razak-Jacelon algebra as one should be able to do as the constructions for those algebras are essentially the same. 

\begin{Definition}
For every pair of natural numbers  $n$ and $n'$ with   $n|n'$ and $a = n'/n -1 >0$, recall that the \textbf{building block algebras} $A(n,n')$ are of the form

$$ A(n,n') = \{f\in C(\interval,M_{n'}): f(0) = \text{diag}(c,c, \hdots,c, 0) \text{ and} f(1)= \text{diag}(c,c,\hdots c) \text{ where } c\in M_n\}.$$
\end{Definition}

\begin{Definition}\label{buildingblockmodels}
For every pair of natural numbers  $n$ and $n'$ with   $n/n'$ and path $u_t:\interval \to U_n$, let $G^u(n,n')$ denote the groupoid obtained using Proposition \ref{cutandpaste} which is a subquotient of $G_{n'}^u(\interval)$ whose $C^*$-algebra is $A(n,n')$. For the constant path $u_t:\interval \to \{\text{id}_n\}$, we denote the groupoid $G^u(n,n')$ by $G(n,n')$. We call the groupoids $G^u(n,n')$  \textbf{building block groupoids}.
\end{Definition}

\begin{Definition}[Jacelon \cite{J}]
The Razak-Jacelon algebra is the inductive limit of any sequence $A_1 \xrightarrow{\phi_1} A_2 \xrightarrow{\phi_2} A_3 \xrightarrow{\phi_3}  \hdots $, where $A_n = A(m_n,(a_n+1)m_n)$ is a building block  algebra such that the connecting morphism $\phi_{m,n} = \phi_{n-1} \circ \hdots \circ \phi_{m+1} \circ \phi_m:A_m\to A_n $ is an injective morphism of the form 

\begin{center}
$   \phi_{m,n}(f) = 
u^*\begin{pmatrix} 
f\circ \xi_1 & 0 & 0 & \hdots & 0 \\
0 & f\circ \xi_2 & 0 & \hdots & 0 \\
0 & 0 & f\circ \xi_3 & \hdots & 0 \\
\vdots & \vdots & \vdots & \ddots & 0\\
0 & 0& 0 & \hdots & f \circ \xi_k
\end{pmatrix}u
$
\end{center}

where $f\in A_m$, $u$ is a continuous path of unitaries in $U_{M_{p_nq_n}}$, $\xi_i$ is a sequence of continuous paths in $\interval$, each satisfying 

$$|\xi_i(x) -\xi_i(y)| \leq 2^{n-m} \hspace{1cm} \forall x,y\in \interval. $$

and 
$$\bigcup_{i=1}^k \xi_i(\interval) = \interval$$
\end{Definition}

\begin{Theorem}\label{RJ}   
There exists a generalized inverse sequence $\{H_i= G^{u_j}((n_j,(a_j+1)n_j)),\psi^i_j\}$ of building block groupoids  whose induced direct system of $C^*$-algebras is equal to the directed system constructed in the proof of Proposition 3.1 in \cite{J}. It follows that the groupoid $C^*$-algebra of the generalized inverse limit $\varprojlim_i H_i$ is the Razak-Jacelon algebra $\mathcal{W}$. Moreover, $G_W$ is an \'etale equivalence relation on a locally compact and non-compact second countable space with covering dimension at most 1..
\end{Theorem}

\begin{proof}

Let $H_1$ be a building block groupoid model (see Definition \ref{buildingblockmodels}) for $A_1$ as in the first line of the proof of Proposition 3.1 in \cite{J} and suppose that $H_j$ and $\psi^k_j$ have been defined for $1\leq j\leq k \leq l$ so that $H_j$ is a building block model for $A_j$ and  the induced map of $\psi^k_{j}$ is equal to the bonding map $\varphi_{jk}$ defined by Jacelon in the proof of Proposition 3.1 in  \cite{J}.

Let $a_l = n_l'/n_l-1$ and notice that $a_l >0$. Let $b=2a_l +1$, $n_{l+1} = bn_l$,  and $m=2b$.

For any natural numbers $k|l$, we define embeddings $\iota_j:M_{k}\hookrightarrow M_l \cong M_{l/k}(M_{k})$ for $j=1,2, \hdots l/k$ where $\iota_j$ is the defined by

\begin{center}
$  \iota_j(T)\hspace{.1cm} = \hspace{.1cm} 
\bordermatrix{~ & 1 & 2 & \hdots & j & \hdots & l/k  \cr
              1& 0 & 0 & \hdots & 0 & \hdots & 0 \cr
              2 & 0 & 0 & \hdots & 0 & \hdots & 0 \cr
              \vdots & \vdots  & \vdots & \vdots & \vdots & \vdots & \vdots \cr
              j & 0 & 0 & \hdots & T  & \hdots  &0  \cr
              \vdots & \vdots  & \vdots & \vdots & \vdots & \vdots & \vdots \cr
              l/k & 0  & 0 & 0 & 0 & 0 & 0 \cr}
$
\end{center}

Next define a path $W_t$ of unitaries in $M_{p_{i+1}q_{i+1}}$ as follows.  Let $\eta_j$ be some element of the standard basis element of $\CC^{p_{i+1}q_{i+1}}$ and choose $n$ to be the unique integer such that $np_iq_i \le j < (n+1)p_iq_i$. We first define $k$ paths $\xi_i:\interval\to \interval$ by


\begin{align*} 
   \xi_i(t) =   \begin{cases} 
      t/2 & 1\leq i\leq b  \text{ and } \\    
      1/2 & r_1 i = b+1 \text{ and } \\         
      (t+1)/2 &  b+1 < i < 2b \text{ and }  \\
   \end{cases}
\end{align*}

and then we define a path of unitaries $W_t$ by $$W_t(\eta_j) = \iota_{n}(U_{\xi_n(t)})(\eta_j).$$

Use Proposition \ref{cutandpaste} to construct a  groupoid $K$ which is a subquotient of $G^W_{n_{i+1}'}$  whose $C^*$-algebra is the subalgebra of $C(\interval, M_{n_{i+1}'})$ consisting of functions $f$ such that

\begin{center}
$  f(0) \hspace{.1cm} = \hspace{.1cm} 
\bordermatrix{~ & 1 & 2 & \hdots & b & b+1& \hdots & 2b  \cr
              1&  \text{diag}(M,M,\hdots,M) & 0 & \hdots & 0 & 0 & \hdots & 0 \cr
              2 & 0 & \text{diag}(M,M,\hdots,M) & \hdots & 0 & 0&  \hdots & 0 \cr
              \vdots & \vdots  & \vdots & \vdots & \vdots & \vdots & \ddots & \vdots \cr
              b & 0 & 0 & \hdots & \text{diag}(M,M,\hdots,M)  & 0&\hdots  &0  \cr
              b+1 & 0 & 0 & \hdots & 0 & N & \hdots  &0  \cr
              \vdots & \vdots  & \vdots & \vdots & \vdots & \vdots & \ddots & \vdots \cr
              2b & 0  & 0 & 0 & 0 & 0 & 0 & N\cr}
$
\end{center}

and 

\begin{center}
$  f(1) \hspace{.1cm} = \hspace{.1cm} 
\bordermatrix{~ & 1 & 2 & \hdots & b+1 & b+2& \hdots & 2b  \cr
              1& N & 0 & \hdots & 0 & 0 & \hdots & 0 \cr
              2 & 0 & N & \hdots & 0 & 0&  \hdots & 0 \cr
              \vdots & \vdots  & \vdots & \vdots & \vdots & \vdots & \ddots & \vdots \cr
              b+1 & 0 & 0 & \hdots & N  & 0&\hdots  &0  \cr
              b+2 & 0 & 0 & \hdots & 0 & \text{diag}(M,M, \hdots, M,0_{n_i}) & \hdots  &0  \cr
              \vdots & \vdots  & \vdots & \vdots & \vdots & \vdots & \ddots & \vdots \cr
              2b & 0  & 0 & 0 & 0 & 0 & 0 & \text{diag}(M,M, \hdots, M,0_{n_i}) \cr}
$
\end{center}
for fixed matrices $M \in M_{n_i}$ and $N\in M_{n_i'}$.
Notice that $K$ is just a Haar system preserving quotient of the disjoint union of $2b$ groupoids where the first $b$ of them are equal to the subgroupoid of $H_i$ given by restricting the interval coordinates to the values $0\leq t\leq 1/2$, but stretched by a factor of $2$ so as to make them bona fide interval groupoids. The $b+1$'st groupoid is exactly equal to the $1/2$, level of $H_i$, but thickened up by the interval in the obvious way. The last $b-1$ groupoids are equal to the subgroupoid of $H_i$ given by restricting the interval coordinates to the values $1/2\leq t\leq 1$, but again these are stretched by a factor of $2$ so as to make them bona fide interval groupoids. $K$ is the quotient by identifying the 0-coordiantes of the first b groupoids and then identifying the 1-coordinate of the last b groupoids. The first $b$ groupoids maps cannonically to $H_i$ by the morphism which puts them into the first half of the groupoid $H_i$. The $b+1$'st groupoid maps to $H_i$ by squashing the interval coordinates to a point and then placing the groupoid at the $1/2$-level of $H_i$ and, lastly, the last $b-1$ groupoids maps to $H_i$ by placing them cannonically into the subgroupoid whose interval coordinates are between $1/2$ and $1$. Let $\Phi$ denote the map just described and notice that $\Phi$ is Haar system preserving and, moreover, the morphism induced by $\Phi$ is of the form

\begin{center}
$   f \to 
\begin{pmatrix} 
f\circ \xi_1 & 0 & 0 & \hdots & 0 \\
0 & f\circ \xi_2 & 0 & \hdots & 0 \\
0 & 0 & f\circ \xi_3 & \hdots & 0 \\
\vdots & \vdots & \vdots & \ddots & 0\\
0 & 0& 0 & \hdots & f \circ \xi_k
\end{pmatrix}
$
\end{center}

As noted by Jacelon in the proof of Proposition 3.1 in \cite{J}, there exists a path of permutation unitaries $u_t$ in $M_{(b+1)n_2}$ such that conjugation by $u_t$ defines an embedding of $C^*{K}$ in $A(n_{i+1},n_{i+1}')$. It follows that $K$ is groupoid isomorphic (via a Haar preserving groupoid isomorphism) to a subquotient $K'$ of subquotient $H_{i+1}$ of $G^{uW}_{n_{i+1}'}$ whose groupoid $C^*$-algebra is exactly $A(n_{i+1},n_{i+1}')$. Let $\psi:K\to K'$ be the groupoid isomorphism and let $q:H_{n_{i+1}'}\to K'$ be the partial quotient map. 

Observe now that the partial morphism $\phi^{i+1}_{i}:H_{i+1}\to H_i$ given by the composition $\Phi \circ \psi\circ q$ has exactly the morphism $\phi_{i,i+1}$ defined in the proof of Proposition 3.1 in \cite{J}. It follows that the generalized inverse system $\{H_i= ,\phi^i_j\}$ is a groupoid model for the inductive system built by Jacelon in Proposition 3.1 of \cite{J}.

\end{proof}


\subsection{Remarks on Self Absorption}

The reader should reference Section \ref{Examples} for the definition and basic examples of self absorbing groupoids.

Unfortunately, being that our groupoid model $G_\mathcal{Z}$ for $\mathcal{Z}$ is isomorphic to the model in \cite{Li}, the object space $\G0_{\mathcal{Z}}$ is 1-dimensional and hence the object space of $G_{\mathcal{Z}}\times G_{\mathcal{Z}}$ will be 2-dimensional. The covering dimension of $G_\mathcal{Z}^\infty$ will be infinite. For purely topological considerations, a groupoid $G$ is self absorbing  only if $\G0$ has covering dimension $0$ or $\infty$. The authors would like to thank Kang Li for pointing out that Xin Li  in \cite{Li} shows that no groupoid model for $\mathcal{Z}$ can have object space with covering dimension zero. To see this, notice that any groupoid model $G$ for $\mathcal{Z}$ must be \'etale with compact object space (as $\mathcal{Z}$ is unital) and so $C(\G0)$ embeds into $C^*(G) = \mathcal{Z}$. However, $C(\G0)$ contains lots of projections if $\G0$ is zero-dimensional and this cannot be $\mathcal{Z}$ is projectionless.

We clearly have a self absorbing groupoid model for $\mathcal{Z}$ given by the groupoid model for $\bigotimes_{\NN}\mathcal{Z}$ using $G_{\mathcal{Z}}$ and the techniques of Subsection \ref{infinitetensorpowers}.  Using this idea, it does follow that every self absorbing $C^*$-algebra which admits a groupoid model does admit a self absorbing groupoid model, but it remains to be seen how useful such models are.

\section*{Appendix A: Full Gelfand Duality in the Commutative Case}\label{gelfand}

The purpose of this appendix is show how our concept of partial morphisms allows one to fix the gap that exists with Gelfand duality, namely, the fact that the inclusion of ideals cannot always be modeled using pullback maps of proper continuous functions. Even though there are models of the category of commutative $C^*$-algebras using topological spaces and partial maps, we believe it is important to note  that our extension of the Gelfand duality functor to groupoids and partial morphisms naturally completes the classical duality. As this appendix is of independent interest, we make an effort to make it self contained and independent of the rest of the paper. Recall that if $A$ is a $C^*$-algebra then an \textbf{approximate unit} for $A$ is a net $\{e_\alpha,\alpha\in \Delta\}\subset A$ of self adjoint elements such that for every $a\in A$ we have $\lim_{\alpha} |a-e_\alpha a| = \lim_\alpha |a-ae_\alpha| = 0$.

The Gelfand duality functor sends every locally compact space $X$ to the $C^*$-algebra $C_0(X)$ and to every proper continuous function $f:X\to Y$ the pullback morphism $f^*:C_0(Y)\to C_0(X)$. It is a much celebrated result that each commutative $C^*$-algebra is of the form $C_0(X)$ for a unique locally compact space $X$ and every  morphism that maps approximate units to approximate units is induced by a unique proper and continuous function. This success perhaps hides another very important class of morphisms of commutative $C^*$-algebras that actually can be modeled at the level of locally compact spaces, namely the $*$-morphisms such that the image does not contain an approximate unit for the codomain algebra. 

Notice that the inclusion $U\subset X$ of an open subspace induces a nonunital embedding $C_0(U) \hookrightarrow C_0(X)$ given by extending continous functions in $C_0(U)$ to be $0$ on $X\setminus U$. This is a perfectly natural $C^*$-morphism to consider. Even more generally, we can consider induced morphisms that come from both extension and pullbacks if we consider \textbf{partial proper continuous maps} defined these as follows: Let $X$ and $Y$ be locally compact Hausdorff spaces. A partial proper continuous function from $X$ to $Y$ is a pair $(f,U)$ where $U
\subset X$ is open  and $f:U\to Y$ is a proper and continuous function. Notice that the pullback $f^*:C_0(Y)\to C_0(U)$ composed with the inclusion $C_0(U)\to C_0(X)$ just described defines a $*$-morphism from $C_0(Y)$ to $C_0(X)$ and it does not need to be unital or a $*$-embedding. We will call the composition $ C_0(Y) \overset{f^*}\to  C_0(U) \xrightarrow{\text{canonical}} C_0(X)$ the \textbf{induced morphism}. 

Here is an example: Notice that there is only one continuous map from the discrete space $\{x_1,x_2\}$ to the discrete space $\{y\}$ and the pullback morphism is given by $\lambda \to (\lambda,\lambda)$ mapping $\CC \to \CC\oplus \CC$. It is \emph{impossible} for the morphism $\CC \to \CC\oplus \CC$ given by $\lambda \to (\lambda,0)$ to be the pullback of a continuous function $f:\{x_1,x_2\} \to \{y\}$. But this morphism is induced by the partial mapping which maps $x_1 $ to $y$ and does nothing to $x_2$.

Let $\mathcal{T}$ denote the category of locally compact spaces with partial proper continuous functions; i.e. it is the full subcategory of $\mathcal{G}$ consisting of locally compact spaces with partial morphisms between them. Let $Comm$ denote the category of commutative $C^*$-algebras with $*$-homomorphisms.

\begin{theorem*}\label{gelfandinfull}
The extension of the functor $\Gamma_c$, which we will denote by $\Gamma_c^e$, that maps each locally compact space $X$ to $C_0(X)$ and sends each partial proper and continuous map to its induced morphism is an opposite equivalence of the categories $\mathcal{T}$  and $Comm$.
\end{theorem*}
\begin{proof}
Let $A= C_0(Y)$ and $B= C_0(X)$ be commutative $C^*$-algebras and let $\phi:A\to B$ be a $*$-morphism. Notice that $C:=\phi(A)\subset B$. We will show that $C$ determines an open subset $U_A\subset X$. Let $\overline{C}$ denote the closure of $C$ in $\ell^\infty(X)$ (after embedding $C_0(X)$ into $\ell^\infty(X)$ via the universal representation) and notice that the s.o. limit (i.e. limit in pointwise convergence topology) of an approximate identity for $C$ converges to the unit $1_C$ of $\overline{C}$. $1_C$ is a projection in $\ell^{\infty}(X)$ and hence corresponds to the characteristic function of some set $U\subset X$. Notice that $U$ is open in $X$ as it is the union of the supports of an approximate identity for $C$. Notice now that $C$ is a $C^*$-subalgebra of $C_0(U)$, when viewed as a subalgebra of $C_0(X)$ by extending functions to be $0$ outside of $U$, and contains an approximate unit and hence there must exist a unique proper continuous mapping $f:U\to Y$ such that $f^*:C_0(Y)\to C_0(U)$ is equal to $\phi:A \to C_0(Y) \subset C_0(X)$. It is straightforward to see that $U$ is the unique subset of $X$ such that $C\subset C_0(U) \subset C_0(X)$ with the first of those inclusions being an approximately unital embedding. The uniqueness of $f$ is by Gelfand duality.
\end{proof}

\begin{remark}
One can say that a large part of this paper is better understanding the $C^*$-algebraic interpretation of groupoids. Because a morphism $\phi:A\to B$ of algebras are really of the form $\phi:A \to \phi(A)$ composed with the inclusion $\phi(A) \hookrightarrow B$, it makes sense that this behavior should be reflected at the topological (or groupoid) level.
\end{remark}

\section*{Appendix B: Increasing Unions of Measure Spaces}\label{measureappendix}

As the proof of Theorem \ref{mainthm.limitexists} shows, inverse limits of  inverse systems in $\mathcal{G}$ are actually increasing unions of inverse limits. As is usual, the hardest part of approximation of groupoids with Haar systems is approximating the Haar systems. We have shown in the proof of Theorem A in \cite{AG} how to take inverse limits of regular Radon measures (which we will replicate in the proof of Theorem \ref{mainthm.limitexists}). The purpose of this appendix is to show how to define a regular Radon measure on an increasing union of regular Radon measures spaces.

\begin{definition}\label{directmeasures}
We say that $(X_\alpha, \Omega_\alpha,\mu_\alpha,p^\alpha,\beta,A)$ is a \textbf{direct system of Borel measure spaces} if
\begin{enumerate}
    \item $A$ is directed
    \item  $X_\alpha$ is a locally compact Hausdorff space for each $\alpha$.
    \item $\mu_\alpha$ is a regular Radon measure for each $\alpha$.
    \item $p^\alpha_\beta:X_\alpha\to X_\beta$ is an inclusion of an open subset for all $\beta\ge\alpha$.
    \item \label{compatiblemeasures} $(p^\alpha_\beta)_*\mu_\alpha = \mu_\beta|_{p^\alpha_\beta(X_\alpha)}$ where $\mu_\beta|_{p^\alpha_\beta(X_\alpha)}$ represents the measure restricted to the submeasure space $p^\alpha_\beta(X_\alpha)$.
    \item $p^\alpha_\alpha = \text{id}_{X_{\alpha}}$ for all $\alpha$.
    \item $p^\alpha_\beta\circ p^\beta_\gamma = p^\alpha_\gamma$ for all $\alpha \ge \beta \ge \gamma$.
\end{enumerate}
\end{definition}

\begin{proposition}\label{getthemmeasurestoworkout}
Let $(X_\alpha, \Omega_\alpha,\mu_\alpha,p^\alpha,\beta,A)$ be a direct system of Borel measure spaces. Denote $X := \bigcup_{\alpha} X_\alpha$ in  and $\Omega$ be the smallest $\sigma$-algebra generated by  $\bigcup_{\alpha}\Omega_\alpha$ . $\Omega$ is contained in the Borel sigma-algebra on $X$ and, moreover, contains all compact  subsets of $X$. The union of the measures $\mu_\alpha$ (when viewed as partial functions in $\Omega\times [0,\infty]$ ) extends uniquely to a measure $\mu'$ on $X$. There exists a unique regular Radon measure $\mu$ on $X$ such that integration of $f\in C_c(X)$ against $\mu'$ is equal to integration against $\mu$.
\end{proposition}
\begin{proof}
Let $p_\alpha:X_\alpha \to X$ be the inclusion mappings. It is clear that $X$ is locally compact and Hausdorff by the requirement that it is the union of open locally compact Hausdorff subspaces.

Recall that the direct limit of sigma-algebras is just the sigma-algebra generated by the images of the $\sigma$-algebras of the pieces in the inverse system. It is clear that the sigma-algebra generated by the images of the Borel subset of the pieces is contained in the Borel subsets of $X$. However, as we are only allowed to take countable sums of sets coming from $\bigcup_{\alpha}\Omega_\alpha$, we only get open sets which are countable sums of open sets. It is easy to see that every compact set in $X$ is the image of a compact subset of $X_\alpha$ for some $\alpha$, so $\Omega$ clearly contains all the compact subsets of $X$.

Let $\Omega$ denote the Borel subsets of $X$ and  $\mu'$ be the union of the measures $\mu_\alpha$ (viewed as partial functions in $\Omega\times [0,\infty]$. We first extend $\mu'$ to a function $\mu:\Omega\to [0,\infty]$. Let $B \in \Omega \setminus \cup_{\alpha}p_\alpha(\Omega_\alpha)$. Then $B= \cup_{i=1}^{\infty} B_{\alpha_i}$ where $B_{\alpha_i}\in \Omega_{\alpha_i}$. We define $\mu(B) = \sum_{i=1}^{\infty}\mu_{\alpha_i}(A_{\alpha_i})$. The reader can check easily that this is well defined  by condition \ref{compatiblemeasures} of Definition \ref{directmeasures}. To show that $\mu$ is sigma-additive, let $\{C_{i}\}$ be a disjoint countable subset of $\Omega$.  Notice that $\sum_{i=1}^{\infty} \mu(C_i) =  \sum_{i,j=1}^{\infty} \mu_{\alpha_{j}}(C_i\cap p_{\alpha_j}(X_{\alpha_{j}}))$. As $\mu_{\alpha_{j}}$ is sigma-addtive for each $j$, we have that  $\sum_{i=1}^{\infty} \mu_{\alpha_{j}}(C_i\cap p_{\alpha_j}(X_{\alpha_{j}})) = \mu_{\alpha_j}(\bigcup_{i=1}^{\infty}C_i\cap p_{\alpha_j}(X_{\alpha_{j}}))$. Hence we have
$$\sum_{i=1}^{\infty} \mu(C_i) = \sum_j^{\infty} \sum_i ^\infty \mu_{\alpha_j}(C_i\cap p_{\alpha_j}(X_\alpha)) = \sum_{j=1}^\infty \mu_{\alpha_j}(\bigcup_{i=1}^{\infty}C_i\cap p_{\alpha_j}(X_{\alpha_{j}})) = \mu (\bigcup_{i} C_i).$$

The fact that $\mu$ is positive and inner regular is clear. Notice that $\mu$ is locally finite and outer regular because $p_\alpha(X_\alpha)$ is open in $X$ for all $\alpha$ and each $\mu_\alpha$ is locally finite and outer regular. One can show that integration against $\mu$ induces a positive linear functional on $C_c(X)$.  The last assertion follows from the Riesz-Markov-Kakutani theorem.
\end{proof}

\section{Acknowledgements}

 The authors are thankful for many helpful conversations with Elizabeth Gillaspy. Her most helpful insights really helped add a greater depth to both the author's understandings of the subject material. The first author would like to thank Adam Dor-on and Jurij Vol\v ci\v c for many helpful discussions during the modeling using partial morphisms stage of the project.  We are grateful to Kang Li for his suggestions to include the covering dimension of our object space and for bringing \cite{Li} to our attention. We are deeply indebted to Claude Schochet for his many comments and questions; they really helped to elevate the accuracy and depth of this paper.

\end{document}